\documentclass{amsart}

\usepackage[all]{xy}

\usepackage{latexsym}

\usepackage{amssymb}
\usepackage{amsfonts}
\usepackage{amscd}
\usepackage{amsmath,amsthm}

\newtheorem{lemma}{Lemma}[section]
\newtheorem{proposition}[lemma]{Proposition}
\newtheorem{theorem}[lemma]{Theorem}
\newtheorem{corollary}[lemma]{Corollary}

\newtheorem{defn}[lemma]{Definition}

\newtheorem{prop}[lemma]{Proposition}
\newtheorem{thm}[lemma]{Theorem}
\newtheorem{cor}[lemma]{Corollary}

{

}

\theoremstyle{definition}

\newtheorem{definition}[lemma]{\sl Definition}

\newtheorem{notn}[lemma]{Notation}

\theoremstyle{remark}

\newcommand{\Hom}{\operatorname{Hom}}

\newcommand{\End}{\operatorname{End}}
\newcommand{\Ext}{\operatorname{Ext}}
\newcommand{\Cdim}{\operatorname{Cdim}}
\newcommand{\uExt}{\underline{\operatorname{Ext}}}
\newcommand{\ra}{\operatorname{\rightarrow}}
\newcommand{\uotimes}{\otimes}
\newcommand{\uTor}{\operatorname{Tor}}
\newcommand{\cM}{\mathcal{M}}

\numberwithin{equation}{section}

\begin{document}

\title{Species and non-commutative $\mathbb{P}^1$'s over non-algebraic bimodules}
\author{D. Chan}
\address{University of New South Wales}
\email{danielc@unsw.edu.au}

\author{A. Nyman}
\address{Western Washington University}
\email{adam.nyman@wwu.edu}

\begin{abstract}
We study non-commutative projective lines over not necessarily algebraic bimodules. In particular, we give a complete description of their categories of coherent sheaves and show they are derived equivalent to certain bimodule species. This allows us to classify modules over these species and thus generalize, and give a geometric interpretation for, results of C. Ringel \cite{species}.
\end{abstract}

\maketitle

\pagenumbering{arabic}

Throughout this paper, $K,K_{0}$, and $K_{1}$ denote fields of characteristic $\neq 2$.

\section{Introduction}

The theory of tilting bundles has grown remarkably over the last few decades, connecting two otherwise unrelated fields of mathematics, algebraic geometry and representation theory. The simplest example occurs with the projective line $\mathbb{P}^1 = \mathbb{P}(K^2) $ which is derived equivalent to the Kronecker algebra or species
$\Lambda =
\left(
\begin{smallmatrix}
K &  K^2\\ 0 & K
\end{smallmatrix}
\right)$.
Since both ${\sf Coh } \mathbb{P}^1$ and $\Lambda$ are hereditary, this means that the indecomposable $\Lambda$-modules are in bijective correspondence with indecomposable sheaves on $\mathbb{P}^1$. More precisely, the indecomposable irregular modules correspond to the line bundles on $\mathbb{P}^1$ whilst the regular modules correspond to the torsion sheaves.

In the 1970's Dlab and Ringel \cite{dlabringel}, \cite{species} studied generalizations of the Kronecker algebra above where $K^2$ is replaced by a $K_{0}-K_{1}$-bimodule $V$ which is constant dimension two on each side (see Section~\ref{section.twosided}). Corresponding bimodule generalizations $\mathbb{P}^{nc}(V)$ of the projective line appeared much later with the work of Patrick \cite{P}, Van den Bergh \cite{vandenbergh}, and Nyman \cite{nyman}. Here one defines a non-commutative version of the symmetric algebra $A = \mathbb{S}^{nc}(V)$, which is a $\mathbb{Z}$-indexed algebra closely related to Dlab and Ringel's preprojective algebra \cite{dlabringel2}. Using the non-commutative projective geometry of Artin-Zhang \cite{az}, we may then define $\mathbb{P}^{nc}(V) =  {\sf Proj}\,A$ to be a certain quotient category of the category of graded $A$-modules (see Section~\ref{section.derived}). However, most research to date has concentrated on the case where the bimodule $V$ is {\em algebraic} in the sense that there exists a common subfield $k$ of $K_{0}$ and $K_{1}$ which acts centrally on $V$ and $K_{i}/k$ is finite. This hypothesis is usually included because it guarantees Hom-finiteness and thus Serre duality.

The goal of this paper is to study non-commutative projective lines and bimodule species without assuming the algebraic hypothesis. In particular, we will re-interpret and extend classic results  of Ringel \cite{species} by exploring derived equivalences between $\mathbb{P}^{nc}(V)$ and the corresponding bimodule species
$$\Lambda =
\left(
\begin{smallmatrix}
K_{0} & V \\ 0 & K_{1}
\end{smallmatrix}
\right).
$$
It seems that the non-commutative projective line is the more amenable object to study and hence our point of view is that non-commutative projective geometry provides a profitable way to study and understand bimodule species. This is partly because many of the familiar notions of algebraic geometry (torsion, dimension, Hilbert functions) extend to sheaves on $\mathbb{P}^{nc}(V)$. For example, each unit $e_i\in A_{ii}, i \in \mathbb{Z}$ gives rise to a line bundle $\mathcal{A}_i$ which is the analogue of $\mathcal{O}(-i)$. We have the following generalization of Grothendieck's splitting theorem which recovers some results in \cite{dlabringel} and \cite{dlabringel2}, and generalizes \cite[Theorem 3.14]{nyman}.

\begin{thm} \label{thm.1}
Any coherent sheaf on $\mathbb{P}^{nc}(V)$ is a direct sum of a torsion free sheaf with a torsion sheaf. Every torsion free sheaf is a direct sum of $\mathcal{A}_i$. In particular, the indecomposable irregular $\Lambda$-modules are induced by multiplication in the non-commutative symmetric algebra and are of the form
$$ A_{i0} \otimes_K V \ra A_{i1} \quad \text{or} \quad
A_{0i-2}^* \otimes_K V \ra A_{1i-2}^*.$$
\end{thm}

It is in the study of regular $\Lambda$-modules or equivalently, torsion sheaves on $\mathbb{P}^{nc}(V)$, that non-commutative projective geometry seems to have something really new to contribute, namely, the notion of a point scheme or commutative locus. We know of no counterpart in the theory of finite dimensional algebras. In our context, this corresponds to the fact that there is a naturally defined normal family of elements $g = \{g_i\} \in A_{ii+\delta}$ where $\delta = 1$ when $V$ is non-simple and $\delta = 2$ when $V$ is simple. Geometrically, we think of $g=0$ as defining a closed subscheme of $\mathbb{P}^{nc}(V)$ and its complement as being affine open. This allows us to geometrically interpret the next result.
\begin{thm}  \label{thm.2}
The category of regular $\Lambda$-modules and the category of torsion coherent sheaves on $\mathbb{P}^{nc}(V)$ are both equivalent to the product category ${\sf fl}\, A[g^{-1}]_{00} \times {\sf T}$ where ${\sf fl}\, A[g^{-1}]_{00}$ is the category of finite length $A[g^{-1}]_{00}$-modules and ${\sf T}$ is uniserial. The decomposition corresponds to the $g$-torsion free and $g$-torsion subcategories of ${\sf Coh}\, \mathbb{P}^{nc}(V)$.
\end{thm}
\noindent
For $V$ non-simple, this is an old result of Ringel's \cite[Section 7.4]{species} and he states quite explicitly that the simple case seems hard. Furthermore, the regular modules in ${\sf T}$ can be ``read off'' the non-commutative symmetric algebra, just as in the commutative case. There are also explicit descriptions of regular modules coresponding to the finite length $A[g^{-1}]_{00}$-modules in Section~\ref{subsec.regular}.

In Section~\ref{sec.torsion} we review $\mathbb{Z}$-indexed algebras and localization theory for them. In Section \ref{section.twosided}, we recall the definition of the non-commutative symmetric algebra from \cite{vandenbergh} and establish the all important Euler exact sequence for them. In Section~\ref{sec.normal} we construct the normal family of elements used in Theorem~\ref{thm.2}. Our construction is suggested by the theory of non-commutative $\mathbb{P}^1$-bundles as developed in \cite{vandenbergh}, but is purely algebraic, in part because the geometric theory of point schemes does not carry over naively to our case. Following standard methodology in non-commutative projective geometry as introduced by Artin-Tate-Van den Bergh \cite{atv1}, we study $A/(g)$ to show that $\mathbb{S}^{nc}(V)$ is noetherian in Section~\ref{section.twisted}. The non-commutative projective line can be defined using the preprojective algebra $\Pi(V)$, but we prefer to use $\mathbb{S}^{nc}(V)$ for a number of reasons. Just like the preprojective algebra, it has good homological properties being Auslander regular of dimension two (proved in Sections~\ref{sec.domain},\ref{sec.Ausreg}) and also a domain. Furthermore, $\Pi(V)$ is in some sense a non-commutative 2-Veronese of $\mathbb{S}^{nc}(V)$ (Proposition \ref{prop.preprojective}) so it is easier to recover the former from the latter. The homological properties of $\mathbb{S}^{nc}(V)$ allow us to conclude that ${\sf Coh}\,\mathbb{P}^{nc}(V)$ is hereditary (Corollary \ref{cor.projAhereditary}), a result which enables us to identify the derived category with the repetitive category. In Section~\ref{section.derived} we compute the cohomology of line bundles on $\mathbb{P}^{nc}(V)$ and use this to invoke tilting theory to show $\mathbb{P}^{nc}(V)$ and $\Lambda$ are derived equivalent. Although a Serre functor is out of the question, we do prove a version of classical Serre duality in Section~\ref{sec.Serre}. Section~\ref{sec.classify} is devoted to classifying sheaves on $\mathbb{P}^{nc}(V)$ and $\Lambda$-modules. In particular, we prove Theorems \ref{thm.1} and \ref{thm.2} above.


\section{$\mathbb{Z}$-indexed algebras and torsion theory}  \label{sec.torsion}

In this section, we recall the notions of indexed algebras and indexed analogues of various ring theoretic concepts such as localization and bimodules.

Let $I$ be a set of indices. Recall (from \cite[Section 2]{quadrics}) that an {\em $I$-indexed algebra} $D$ is a pre-additive category whose objects are indexed by $I$ and denoted (for reasons given below) $\{\mathcal{O}(-i)\}_{i \in I}$, and whose morphisms are denoted $D_{ij} := \operatorname{Hom}(\mathcal{O}(-j), \mathcal{O}(-i))$. If the category is in fact $k$-linear for some field $k$, then we say $D$ is an {\em $I$-indexed $k$-algebra}. We will often abuse terminology and call the ring $D = \oplus_{i,j}D_{ij}$ with multiplication given by composition an $I$-indexed algebra. The case we are most interested in is $I = \mathbb{Z}$. We let $e_{i}$ denote the identity in $D_{ii}$. We note that in the literature, $\mathbb{Z}$-indexed algebras are more often called $\mathbb{Z}$-algebras.

The example to keep in mind comes from projective geometry. Let $X$ be the projective line, or more generally any projective variety embedded in projective space, let $I = \mathbb{Z}$, and let $\mathcal{O}(i)$ denote the $i$-th tensor powers of the tautological line bundle.  Then $D_{ij} = \Hom_{\mathcal{O}_X}(\mathcal{O}(-j), \mathcal{O}(-i))$.

In the indexed setting, unlike the graded one, all objects are graded and it is unnatural to form ungraded versions. For example, let $D$ be an $I$-indexed algebra. A {\em (graded) right $D$-module} is a graded abelian group $M = \oplus_{i \in I} M_i$ with multiplication maps $M_i \times D_{ij} \ra M_j$ satisfying the usual module axioms (see \cite{quadrics} for more details). We let ${\sf Gr }D$ denote the category of graded right $D$-modules and we let $D-{\sf Gr}$ denote the category of graded left $D$-modules.

As for rings, one way to see the symmetry between left and right modules is to introduce the {\em opposite algebra $D^{op}$} which is just the opposite category. Then left $D$-modules correspond to right $D^{op}$-modules in the usual way.

Let $D_{ij}$ be a $\mathbb{Z}$-indexed algebra. We say that it is a {\it domain} if i) for all $i,j$ and non-zero $a \in D_{ij}$, we have that $a$ is a non-zero-divisor, i.e. if $b \in D_{jl}, c \in D_{hi}$ is non-zero, then $ab, ca \neq 0$, and ii) $D_{ii+1} \neq 0$. The second condition will become clear when we construct the indexed ring of fractions.

As noted in \cite{crem}, the theory of Ore sets and Ore localization works fine for $\mathbb{Z}$-indexed algebras (recall that the construction of the derived category by inverting quasi-isomorphisms is such a general case of localization). As usual, a $D$-module is {\it uniform} if any two non-zero submodules have non-zero intersection, or equivalently, any non-zero submodule is essential. Our analogue of the ring of fractions is given in the next result.

\begin{proposition} \label{prop.ringquotients}
Let $D$ be a domain such that $e_iD$ is uniform for every $i$. Then the set of non-zero elements forms a right Ore set. Inverting these elements gives a $\mathbb{Z}$-indexed algebra $Q$ which we call the {\em (right) ring of fractions}.
\end{proposition}
\noindent
\textbf{Remark} Actually, we may weaken the hypotheses on $e_i D$ to only assume they all have the same Goldie rank. We will not need this result.
\begin{proof}
 Suppose $a_{is} \in D_{is}, x_{ir} \in D_{ir} \backslash \{0\}$. We need to find a non-zero $y \in e_s D$ and $b \in e_r D$ such that $a_{is} y =  x_{ir} b$. If $a_{is} = 0$ we need only set $b=0$ (with $y$ arbitrary) whilst if $a_{is} \neq 0$ then the result follows as $a_{is} e_s D, x_{ir} e_r D$ are non-zero submodules of the uniform module $e_i D$.
\end{proof}

If the hypotheses of the proposition hold, we say $D$ {\em has a right ring of fractions}. If also all the $D e_i$ are uniform left modules, then we say $D$ {\em has a ring of fractions}.
Rings of fractions have the usual nice properties.

\begin{proposition}  \label{prop.structureQ}
Suppose $D$ has a right ring of fractions $Q$.
\begin{enumerate}
 \item The $Q_{ii}$ are isomorphic division rings and $Q_{ij}$ is a $Q_{ii}-Q_{jj}$-bimodule which is 1-dimensional on each side and hence gives a Morita equivalence between $Q_{ii}$ and $Q_{jj}$.
 \item Let $M$ be a right $Q$-module. Then we have an isomorphism of $Q$-modules $\phi: M_0 \otimes_{Q_{00}} e_0 Q \ra M$ given by multiplication $M_0 \otimes_{Q_{00}} Q_{0j} \ra M_j$.
\end{enumerate}
\end{proposition}
\begin{proof}
 As in the previous proof, this one follows the usual ring theory proof so we only mention how condition ii) in our definition of a domain enters the picture. Note that since $D$ is a domain, we have $D_{ij} \neq 0$ if $i < j$. Given any non-zero element $d_{ij} \in D_{ij}$, conjugation by $d_{ij}$ gives an isomorphism of $D_{ii}$ with $D_{jj}$.
\end{proof}

We will need the following definition, from \cite[Section 3]{duality}.

\begin{definition} \label{def.B}
Let $D$ be an $I$-indexed algebra.  We let ${\sf Bimod }D-D$ denote the category of $D-D$-bimodules.  Specifically:
\begin{itemize}

\item{}
an object of ${\sf Bimod }D-D$ is a triple
$$
(B=\{B_{ij}\}_{i,j \in I}, \{\mu_{ijk}\}_{i,j,k \in I}, \{\psi_{ijk}\}_{i,j,k \in I})
$$
where ${B}_{ij}$ is an abelian group and $\mu_{ijk}:B_{ij} \otimes D_{jk} \rightarrow B_{ik}$ and $\psi_{ijk}: D_{ij} \otimes B_{jk} \rightarrow B_{ik}$ are group homomorphisms making $B$ a $D$-$D$ bimodule.

\item{}
A morphism $\phi:  B \rightarrow C$ between objects in ${\sf Bimod }D-D$ is a  collection $\phi=\{\phi_{ij}\}_{i,j \in I}$ such that $\phi_{ij}:B_{ij} \rightarrow D_{ij}$ is a group homomorphism, and such that $\phi$ respects the $D-D$-bimodule structure on $B$ and $C$.

\end{itemize}

\end{definition}
As usual, these can be studied via an appropriate definition of an enveloping algebra, as follows. Let $D_I$ be an $I$-indexed $k$-algebra and $D_H$ be an $H$-indexed $k$-algebra. We define an $I \times H$-indexed $k$-algebra $D_I \otimes_k D_H$ by
$$(D_I \otimes_k D_H)_{(i,h)(i',h')} := D_{I,ii'} \otimes_k D_{H,hh'} .$$
If $D$ is a $\mathbb{Z}$-indexed $k$-algebra, then the {\em enveloping algebra} $D^{op} \otimes_k D$ is a $\mathbb{Z}^2$-indexed algebra and $D-D$-bimodules are just right $D^{op} \otimes_k D$-modules.

For each $h \in I$, there is a natural {\em restriction functor} $\operatorname{res}_h:{\sf Bimod}\, D-D \ra {\sf Gr}\, D$ defined by $B \mapsto e_h B$ and similarly for left modules. There is an exact left adjoint $De_h \otimes_k -$ which sends the $D$-module $M$ to the bimodule defined by $(D e_h \otimes_k M)_{ij} = D_{ih} \otimes_k M_j$. In particular, we see that if $I$ is an injective $D-D$-bimodule, then $Ie_h, e_hI$ are injective left and right $D$-modules.

Given a right $D$-module $M$ and $D-D$-bimodule $B$, we can form the tensor product $M \otimes_D B$ (see \cite[Defintion 3.2]{duality}) which is again a right $D$-module, and this is functorial in $M$ and $B$ (a similar construction holds for left modules). Assume now that  $D$ has a ring of fractions $Q$, which we can also view as a $D-D$-bimodule. The natural inclusion map $D \ra Q$ induces a localization map $\lambda_M: M \ra M \otimes_D Q$. If $M$ is a noetherian right $D$-module, we say $M$ is {\it torsion} if $M \otimes_D Q = 0$ and {\it torsion-free} if $\lambda_M$ is injective. Standard arguments using exactness of localization give the following result.

\begin{proposition}  \label{prop.torsionmodules}
\begin{enumerate}
 \item $\sigma M:= \ker \lambda_M$ is torsion and contains every torsion submodule of $M$. It is called the {\em torsion submodule} of $M$.
 \item $M/ \sigma M$ is torsion-free.
\end{enumerate}
\end{proposition}

We need one more standard localization result in the $\mathbb{Z}$-indexed setting.

\begin{prop}   \label{prop.structuretorfree}
Let $M$ be a finitely generated torsion-free right $D$-module. Then there is an embedding $\phi: M \ra (e_i D)^{\oplus n}$ for some $i,n$ such that $\operatorname{coker} \phi$ is torsion.
\end{prop}
\begin{proof}
 Note first that $M \otimes_D Q$ is finitely generated as a $Q$-module so by the structure theory of Proposition~\ref{prop.structureQ} it is isomorphic to $(e_0Q)^n$ for some integer $n$. We may thus assume that $M$ is a submodule of $(e_0Q)^n$ and hence write elements of $M$ as $n$-tuples of left fractions. We wish to find a common left denominator $a_{i0} \in D_{i0}$ for all the elements in $M$, for then $a_{i0} M \subseteq (e_i D)^n$ and since $D$ is a domain, left multiplication by $a_{i0}$ gives the desired embedding $\phi$.

Now $M$ is finitely generated, so we need only show that any two elements of $e_0 Q$ have a common denominator, or equivalently, any two such denominators $x_{r0} \in D_{r0}, y_{s0} \in D_{s0}$ have a common left multiple. $D e_r x_{r0} \cap D e_s y_{s0}$ is a non-zero submodule of the uniform module $D e_0$ so we are done. Note that $\phi \otimes_D Q$ is an isomorphism so $\operatorname{coker} \phi$ is indeed torsion.
\end{proof}

\section{Two-sided vector spaces and non-commutative symmetric algebras} \label{section.twosided}

In this section, we recall basic definitions and facts regarding two-sided vector spaces and non-commutative symmetric algebras. We then establish the Euler exact sequence in the not necessarily algebraic case.

\subsection{Two-sided vector spaces}   \label{subsec.2sided}

By a \emph{two-sided vector
space} we mean a $K_{0}-K_{1}$-bimodule $V$. If $V \neq 0$, then the characteristic field of both $K_0$ and $K_1$ are the same and acts centrally on $V$. In general, we use the symbol $k$ to denote any common subfield of $K_0,K_1$ which acts centrally on $V$. If a central subfield $k$ can be chosen so that $K_{i}/k$ is a finite separable extension, then following Ringel \cite{species}, we say that $V$ is an {\em algebraic bimodule}.  We say $V$ has {\it rank $n$} if $\operatorname{dim}_{K_{0}}(V)
= \operatorname{dim}_{K_{1}}(V)=n$.

In order to describe the next result classifying rank two two-sided vector spaces, we introduce some notation.  If $M_{n}(K_{0})$ denotes the ring of $n\times n$ matrices over $K_{0}$ and $\phi:K_{1}\rightarrow M_n(K_{0})$ is a nonzero homomorphism, then we denote by $K^n_\phi$ the two-sided vector space whose underlying set is $K_{0}^{n}$, whose left action is the usual one, and whose right action is via $\phi$.

We have the following variant of \cite[Theorem 1.3]{P}
\begin{lemma} \label{lemma.twosidedclass}
Suppose $V$ has rank two.
\begin{enumerate}
\item{} If $V$ is non-simple, then $K := K_{0} \cong K_{1}$.  Considering $V$ as a $K-K$-bimodule, we have
\begin{enumerate}
\item{} $V \cong K^{2}_{\phi}$ where $\phi(a)=\begin{pmatrix}
\sigma(a) & 0 \\ 0 & \tau(a) \end{pmatrix}$ and $\sigma, \tau \in
\operatorname{Aut}(K)$, or

\item{} $V \cong K^{2}_{\phi}$ where $\phi(a)=\begin{pmatrix}
\sigma(a) & \delta(a) \\ 0 & \sigma(a) \end{pmatrix}$, $\sigma(a)
\in \operatorname{Aut}(K)$, and $\delta$ is a $(\sigma, \sigma)$-derivation.
\end{enumerate}

\item{} If $V$ is simple, then there exists a degree two field extension $F$ of $K_{0}$ and $K_{1}$ such that $V$ is isomorphic to the two-sided vector space ${}_{0}F_{1}$ with underlying additive group $F$ and left and right actions induced by the field embeddings of $\iota_0:K_{0}\ra F$ and $\iota_1:K_{1}\ra F$, respectively.
\end{enumerate}
\end{lemma}

\begin{proof}
If $V$ is non-simple, then $V$ has a submodule $W$ which is one-dimensional on the right and left.  It follows that $K_{0} \cong K_{1}$, whence the first assertion.  The second assertion follows from \cite[Theorem 1.3]{P}.

Let $k$ be the common characteristic field of $K_0$ and $K_1$. If $V$ is a simple $K_{0} \otimes_{k} K_{1}$-module, there exists a maximal ideal $J$ of $K_{0} \otimes_{k} K_{1}$ such that $V \cong K_{0} \otimes_{k} K_{1}/J := F$.  The result follows.
\end{proof}
\noindent
Several of our results will be proven on a case by case basis depending on the structure of $V$ according to Lemma \ref{lemma.twosidedclass}.

We also need to recall from \cite{vandenbergh}, the notion of left and right dual of a two-sided vector space.  The {\it right dual of $V$}, denoted $V^{*}$, is the $K_1-K_0$-bimodule
$\operatorname{Hom}_{K_{1}}(V_{K_{1}},K_{1})$ with action $(a \cdot \psi \cdot
b)(v)=a\psi(bv)$ for all $\psi \in
\operatorname{Hom}_{K_{1}}(V_{K_{1}},K_{1})$ and $a \in K_{1}, b \in K_0$. Similarly, the {\it left dual of $V$}, denoted ${}^{*}V$, is the $K_0-K_1$-bimodule
$\operatorname{Hom}_{K_{0}}({}_{K_{0}}V,K_{0})$ with action $(a \cdot \phi
\cdot b)(v)=b \phi(va)$ for all $\phi \in
\operatorname{Hom}_{K_{0}}({}_{K_{0}}V,K_{0})$ and $a \in K_1, b \in K_{0}$.  The left and right duals of two-sided vector spaces are clearly functorial.

In case $V$ is simple of rank two, these duals are easily computed. To this end, we fix some useful notation. Let $\iota_0:K_0 \ra F, \iota_1: K_1 \ra F$ be field homomorphisms such that $[F:\iota_j(K_j)] = 2$ for $j=0,1$. We let $\sigma_j$ be the Galois involution generating the Galois group of $F/\iota_j(K_j)$.
Note that the maps $\iota_0,\iota_1$ make $F$ both a $K_0-K_1$-bimodule and a $K_1-K_0$-bimodule. We will use the symbols $_0F_1$ and $_1F_0$ to denote these bimodules.  There is a (reduced) trace map
$$ \text{tr}_j: F \ra K_j: a \mapsto \iota_j^{-1}\left( \frac{1}{2}( a + \sigma_j(a)) \right) .$$
The proof of the following lemma is immediate.
\begin{lemma}  \label{lemma.dualsimple}
With the above notation, the trace map gives a non-degenerate trace pairing and hence an isomorphism of $K_1-K_0$-bimodules $_0F_1^* := \Hom_{K_1}(_0F_1,K_1) \cong\ _1F_0$. Explicitly, this is
$$ _1F_0 \xrightarrow{\sim} \ _0F_1^*: a \mapsto [b \mapsto \text{tr}_1(ba)] .$$
We similarly also have a bimodule isomorphism
$$ _1F_0 \xrightarrow{\sim} \ _0^*F_1: a \mapsto [b \mapsto \text{tr}_0(ab)] .$$
\end{lemma}

\subsection{Non-commutative symmetric algebras}  \label{subsec.SncV}

It will be useful to introduce the following notation.

\begin{notn}  \label{not.Ki}
For $i \in \mathbb{Z}$, we let $K_i = K_0$ if $i$ is even and $K_i = K_1$ if $i$ is odd. We will often drop the subscript if it is clear from the context. Furthermore, when $V$ is non-simple so $K_0 \cong K_1$ by Lemma~\ref{lemma.twosidedclass}, we fix such an isomorphism and so identify $K_0 = K_1 = K$.
\end{notn}

Given a two-sided vector space $V$, we set
$$
V^{i*}:=
\begin{cases}
V & \text{if $i=0$}, \\
(V^{(i-1)*})^{*} & \text{ if $i>0$}, \\
{}^{*}(V^{(i+1)*}) & \text{ if $i<0$}.
\end{cases}
$$
This is a $K_i-K_{i+1}$-bimodule, which given our convention above, we will sometimes refer to as just a $K-K$-bimodule.


\begin{lemma} \label{lemma.ranktwo}
If $V$ has rank two, then $V^{i*}$ has rank two for all $i$ and there is an isomorphism $V \rightarrow V^{**}$.
\end{lemma}
\begin{proof}
If $V$ is non-simple, the assertions follow from the formula for left and right dual appearing in \cite[Section 6.4]{species}.

If $V$ is simple, this follows from Lemma \ref{lemma.twosidedclass} and Lemma \ref{lemma.dualsimple}.
\end{proof}

There is a natural isomorphism $ \operatorname{End}_{K_{i+1}}(V^{i*})   \cong  V^{i*} \otimes  V^{(i+1)*}$.  Furthermore, since $V^{i*}$ is a $K-K$-bimodule, there is a ring homomorphism $K_i \ra \operatorname{End}_{K_{i+1}}(V^{i*})$. Combining these gives a $K-K$-bimodule morphism $K_i \rightarrow V^{i*} \otimes  V^{(i+1)*}$. We sometimes denote the image of this map by $Q_{i}$.

As in the graded case, to define the symmetric algebra, we first define the {\em tensor algebra}. This is the $\mathbb{Z}$-indexed algebra  $T = \oplus T_{ij}$ with
\begin{itemize}
\item{} $T_{ij}=0$ if $i>j$,

\item{} $T_{ii}=K_i$,

\item{} and freely generated by $T_{i i+1}=V^{i*}$,
\end{itemize}
In other words,
$$
T_{ij} = T_{ii+1} \otimes T_{i+1 i+2} \otimes \cdots \otimes T_{j-1j}.
$$
We recall (from \cite{vandenbergh}) that the {\em non-commutative symmetric algebra} of $V$, is the quotient $\mathbb{Z}$-indexed algebra $\mathbb{S}^{nc}(V)=T_{ij}/R_{ij}$ where $R_{ij}$ is the ideal generated by $Q_{i}, \ldots, Q_{j-2}$ if $j \geq i+2$ or is zero otherwise. In the sequel, the ring $\mathbb{S}^{nc}(V)$ will be denoted by $A$.

In the sequel, we will assume $V$ has rank two.  The study of $\mathbb{S}^{nc}(V)$ in case $V$ is an algebraic bimodule with $\operatorname{dim}_{K_{0}}(V)=1$ and
$\operatorname{dim}_{K_{1}}(V)=4$ is carried out in \cite{tsen} and \cite{witt}.

We will routinely identify $A$ with the $\mathbb{Z}$-indexed algebra generated by $V$ and $V^{*}$ in bidegrees $(i,i+1)$ with $i$ even and odd respectively, and with relations induced by $Q_{0}$ and $Q_{1}$.  This follows from a routine computation using Lemma \ref{lemma.ranktwo}, \cite[Lemma 2.4]{nyman} and \cite[Lemma 2.1]{tsen}.

\subsection{The Euler exact sequence} \label{section.ees}
Below we use the following notation, which is valid in light of Lemma \ref{lemma.ranktwo}. We pick a two-sided basis $\{x,y\}$ for $V^{j*}$ and let $\{x^*,y^*\} \in V^{j+1*}$ be the dual left basis and $\{ ^*x, ^*y\} \in V^{j-1*}$ be the dual right basis. Hence $Q_j$ is generated on the left or right by $x \otimes x^* +  y \otimes y^*$ and similarly for $Q_{j-1}$. As in Section~\ref{subsec.SncV}, we let $T$ denote the tensor $\mathbb{Z}$-indexed algebra on $V$ and $A = \mathbb{S}^{nc}(V) = T/R$.

The next result shows that the elements of $V^{j*}$ are non-zero-divisors.

\begin{lemma}  \label{lemma.Vnzd}
Let $0\neq x \in V^{j*}, a \in T_{ij}$ be such that $a \otimes x \in R_{ij+1}$. Then $a \in R_{ij}$.
\end{lemma}
\begin{proof}
We argue by induction on $j-i$, the case $j-i=1$ being clear. Since $R_{ij+1} = R_{ij}\otimes V^{j*} + T_{ij-1} \otimes Q_{j-1}$, we may find $b \in T_{ij-1}$ such that
$$a \otimes x - b \otimes ( \,^*x \otimes x +  \,^*y \otimes y) =
(a  - b \otimes \,^*x) \otimes x -  b \otimes \,^*y \otimes y
  \in R_{ij} \otimes V^{j*} .$$
In particular, we see that both $b \otimes \,^*y , a - b \otimes \,^*x \in R$. By induction, we see then that $b$ lies in $R$ and hence so does $a$.
\end{proof}

We now prove the existence of the Euler exact sequence.

\begin{lemma} \label{lemma.eulerses}
\begin{enumerate}
\item The following equality holds,
$$ R_{ij+1} \otimes V^{j+1*} \cap T_{ij} \otimes Q_j  =  R_{ij} \otimes Q_j .
$$
\item There is the following exact sequence
$$0 \rightarrow A_{ij} \otimes Q_{j} \rightarrow A_{ij+1} \otimes V^{j+1*} \rightarrow A_{ij+2} \rightarrow 0.$$
In particular, the left and right dimensions of $A_{ij}$ equals $j-i+1$ for $j \geq i$.
\item Furthermore, there is an exact sequence of left $A$-modules
$$ 0 \ra A e_j \ra (A e_{j+1})^{\oplus 2} \ra A e_{j+2} \ra A_{j+2,j+2} \ra 0 .$$
\end{enumerate}
\end{lemma}
\begin{proof}
Note that $(2) \implies (3)$ since it shows exactness in degrees $i \leq j$, whilst exactness in degrees $i>j$ is immediately checked from the definition of $A$. Note also that by definition of $A$, the sequence in part (2) is a complex which is exact everywhere except possibly the $A_{ij} \otimes Q_{j}$ term. Exactness there is equivalent to the statement in (1), so we are reduced to proving (1).

We consider an element of the intersection on the left hand side of (1), which can be written as $c \otimes (x \otimes x^* +  y \otimes y^*)$ for some $c \in T_{ij}$. It can also be written in the form
\begin{equation}
c \otimes (x \otimes x^* +  y \otimes y^*) =
(a \otimes x^* + a' \otimes y^*) + (b \otimes x^* + b' \otimes y^*)
\label{eq:euler}
\end{equation}
where $a,a' \in R_{ij} \otimes V^{j*}, b,b' \in T_{ij-1} \otimes Q_{j-1}$. We write
$$ b = b'' \otimes ( ^*x \otimes x +  ^*y \otimes y),
b' = b''' \otimes (^*x \otimes x +  ^*y \otimes y)$$
Substituting back into (\ref{eq:euler}), we find that
$$ R_{ij} \otimes V^{j*} \ni a =
(c- b'' \otimes\, ^*x ) \otimes x - b'' \otimes\, ^*y \otimes y .$$
This shows that $b'' \otimes \,^*y \in R$ so $b'' \in R$ by lemma \ref{lemma.Vnzd}. We also see that $c  - b'' \otimes ^*x \in R$ so $c \in R$ too. It follows that $c \otimes (x \otimes x^* +  y \otimes y^*)   \in R_{ij} \otimes Q_j$ and the lemma is proved.
\end{proof}

\section{Normal elements in $A$}  \label{sec.normal}

In \cite{vandenbergh}, M. Van den Bergh proves that if $X$ and $Y$ are smooth schemes of finite type over a field $k$, the category of graded right modules over the non-commutative symmetric algebra of a $k$-central rank two $\mathcal{O}_{X}-\mathcal{O}_{Y}$-bimodule is noetherian (D. Presotto and Louis de Thanhoffer de V\"{o}lcsey prove a similar result in case the bimodule has left rank four and right rank one \cite{presotto}).  The key notion in Van den Bergh's paper is the point scheme which is obtained by geometric means and unavailable in our setting. Using the point scheme, he constructs a ``projectively commutative'' quotient algebra, namely the twisted homogeneous co-ordinate ring, from which he deduces the ascending chain condition for the non-commutative symmetric algebra. In this section, we show that in our setting this quotient still exists by proving algebraically that the non-commutative symmetrical algebra of a rank two two-sided vector space has a family of normal elements.

Let $D$ be a $\mathbb{Z}$-indexed algebra. A {\it normal family of elements} $g = \{g_i\}$ of degree $\delta$ consists of $g_i \in D_{i,i+\delta}, i \in \mathbb{Z}$ such that $g_i D_{i+\delta,j+\delta} = D_{ij} g_j$ for all $i,j \in \mathbb{Z}$. We let $(g)$ be the ideal generated by the family.

As usual, we let $V$ be a rank two two-sided vector space, $A = \mathbb{S}^{nc}(V)$, and we routinely use the notation in Section~\ref{section.twosided}.
\subsection{Case 1: $V$ is simple} \label{subsection:simplecase}
We retain the notation from Lemma~\ref{lemma.twosidedclass} so $V = \,_0F_1$. Pick $w_j \in F$ so that $w_j^2 \in \iota_j(K_j)$ but $w_{j}$ is not in $\iota_{j}(K_{j})$. Multiplying $w_{j}$ by the appropriate element of $\iota_j(K_j)$, we may also assume that $w_{j}$ is not in either $\iota_{0}(K_{0})$ or $\iota_{1}(K_{1})$.

Since $F/\iota_1(K_1)$ is Galois, $T_{02}(\,_0F_1) := \,_0F_1 \otimes \,_0F_1^* = \,_0F_1 \otimes \, _1F_0 = F \otimes_{K_1} F$ can be considered a $K_1$-algebra and as such it is isomorphic to $\langle \sigma_1 \rangle \times F \cong F \times F$. Furthermore, $\,_0F_1 \otimes \,_0F_1^*$ is also an $F$-bimodule (since $F \otimes_{K_1} F$ is) and as such decomposes as $\,_0F_1 \otimes \,_0F_1^*  = F_{\text{id}} \oplus F_{\sigma_1}$. We examine the component $F_{\text{id}}$ below. Using Lemma~\ref{lemma.dualsimple}, we see that $Q_0 \subset \,_0F_1 \otimes \,_1F_0$ is the $K_0$-central subbimodule generated by the canonical element $h_1 := 1 \otimes 1 + w_1^{-1} \otimes w_1$.

\begin{lemma}\label{lemma.Fcentral}
\begin{enumerate}
\item $FQ_0 = Q_0 F$ so is an $F$-subbimodule of $\,_0F_1 \otimes \,_1F_0$.
\item $FQ_0$ is $F$-central.
\item $\{h_1, g_1:= w_1 h_1 = w_1 \otimes 1 + 1 \otimes w_1\}$ is a simultaneous left and right $K_0$-basis for $F Q_0$.
\end{enumerate}
\end{lemma}
\begin{proof}
For $a \in F$, all statements will follows if we can show that $ah_1 = h_1a$. To this end, we write $a = a' + a''w_1$ where $a',a'' \in \iota_1(K_1)$. Then
$$ a ( 1 \otimes 1 + w_1^{-1} \otimes w_1) =
a' \otimes 1 + a''w_1 \otimes 1 + a' w_1^{-1}\otimes w_1 + a'' \otimes w_1 .$$
Similarly,
$$(1 \otimes 1 + w_1^{-1} \otimes w_1) a = 1 \otimes a' + a'' \otimes w_1 + a'w_1^{-1} \otimes w_1 + a''w_1^{-1} \otimes w_1^2 .$$
The lemma follows since $w_1^2 \in \iota_1(K_1)$ commutes through the tensor.
\end{proof}

Define $h_0:= 1 \otimes 1 + w_0^{-1} \otimes w_0, g_0 = w_0 h_0$ which forms a $K_1$-basis for $FQ_1 = Q_1 F$.

\begin{lemma}\label{lemma.modFQ}
Given $a \otimes b \in \ _0F_1 \otimes\, _1F_0$ we have
$$ a \otimes b \equiv  a \sigma_1(b) \otimes 1 \mod FQ_0.$$
\end{lemma}
\begin{proof}
We write $a = a' + a''w_1, b = b' + b''w_1$ with $a',a'',b',b'' \in \iota_1(K_1)$. Then
\begin{eqnarray*}
a \otimes b & = & a \otimes b' + a' \otimes b''w_1 + a''w_1 \otimes b'' w_1 \\
                 & \equiv & ab' \otimes 1 - a'b''w_1 \otimes 1 - a''b''w_1^2 \otimes 1 \mod FQ_0 \\
                  & = & a(b' - b'' w_1) \otimes 1
\end{eqnarray*}
\end{proof}

\begin{proposition}\label{proposition.gnormal}
The images $\overline{g}_0, \overline{g}_1$ of $g_0,g_1$ in $A$ form a normal pair in the sense that given $a \in F$ there exists $b \in F$ such that $\overline{g}_1 a = b \overline{g}_0$.
\end{proposition}
\begin{proof}
We analyze the equation in the triple tensor product $_0F_1 \otimes \,_1F_0 \otimes\,_0 F_1$. It suffices to show that in this product, $g_1 \otimes a$ lies in
$$ F \otimes g_0 + F \otimes Q_1 + Q_0 \otimes F = F \otimes FQ_1 + Q_0 \otimes F .$$
Using Lemma~\ref{lemma.modFQ} we find
\begin{eqnarray*}
g_1 \otimes a & \equiv & g_1 w_1^{-1}w_1 \sigma_0(a) \otimes 1 \mod F \otimes FQ_1 \\
                     & \equiv & h_1 \otimes \sigma_0^{-1}(w_1) a  \mod F \otimes FQ_1  \\
                     & \in & Q_0 \otimes F
\end{eqnarray*}
\end{proof}
For $i \in \mathbb{Z}$, we will abuse notation and write $g_{i} \in A$ for $\overline{g}_{\bar{i}}$ where $\bar{i}$ is the residue of $i$ modulo 2.

\subsection{Case 2: $V$ non-simple} \label{section.oned}
By Lemma \ref{lemma.twosidedclass}, we know that in this case $K=K_{0} \cong K_{1}$.  Furthermore, we do not change the isomorphism class of $A$ by considering $V$ as a $K-K$-bimodule.  Thus, we may assume without loss of generality that $K_{0}=K_{1}$ and, in the notation of Lemma \ref{lemma.twosidedclass}, that $V = K^{2}_{\phi}$ where $\phi$ is defined in cases (1a) and (1b) of the lemma.  We let $x = (1,0)$, $y=(0,1)$, and we let $x^{*}$ and $y^{*}$ be a dual left basis for $V^{*}$.  Then the bimodules generated by $y$ and $x^{*}$ are invertible bimodules.  Therefore, since $y^{*} y = - x^{*} x$ and $y  y^{*}=-x  x^{*}$ in $A$, the elements
$$
g_{i} = \begin{cases} y \in A_{i i+1} \mbox{ for $i$ even, and} \\ x^{*} \in A_{i i+1} \mbox{ for $i$ odd.} \end{cases}
$$
forms a normal family of degree one.

We summarise our results in the two cases below.
\begin{proposition}  \label{prop.normalfamily}
The family of elements $g = \{g_i\}$ is normal of degree 2 if $V$ is simple and of degree 1 if $V$ is non-simple.
\end{proposition}
We remark that in the non-simple case, $g^2 = \{g_ig_{i+1}\}$ is also a normal family of degree 2. In Van den Bergh's setup, this is the more natural element to consider.

\section{The twisted ring $A/(g)$}  \label{section.twisted}

In this section, we analyze the quotient algebra $A/(g)$ where $g$ is the normal family of elements obtained in the previous section. We show it is the analogue of the twisted homogeneous co-ordinate ring studied in the theory of non-commutative $\mathbb{P}^1$-bundles \cite{vandenbergh} and in particular is noetherian. This in turn allows us to apply the Hilbert basis theorem \cite[Lemma 3.2.2]{vandenbergh} to show that $A$ is noetherian.


We first define the notion of a {\em twisted ring} in the context of $\mathbb{Z}$-indexed algebras. Let $F$ be a field and $\sigma_i \in \operatorname{Aut} F, i \in \mathbb{Z}$. We will sometimes refer to the $\sigma_i$ as {\em twisting automorphisms}. We define the {\em full twisted ring on $F$}  to be the $\mathbb{Z}$-indexed algebra $C = C(F;\sigma)$ with $C_{ij}=F$ and multiplication $m_{ijk}: C_{ij} \otimes_F C_{jk} \ra C_{ik}$ defined by
\begin{equation}  \label{eqn.muC}
 m_{ijk}(a \otimes b ) =
\begin{cases}
a\sigma_{i+1}\sigma_{i+2}\ldots\sigma_{j}(b) & \text{if } i \leq j \\
a\sigma_{i}^{-1}\sigma^{-1}_{i-1}\ldots\sigma_{j+1}^{-1}(b) & \text{if } i > j
\end{cases}
\end{equation}
Now $F$ is noetherian, so this strongly graded ring is also (left and right) noetherian by \cite[Lemma 3.2.3]{vandenbergh}. A {\em twisted ring} (on $F$) is any subring $B$ which equals $C$ in sufficiently large degree, that is, there exists $d \in \mathbb{N}$ such that $B_{ij} = C_{ij}$ whenever $j-i \geq d$.

\cite[Lemma 3.2.1]{vandenbergh} gives the following result.
\begin{proposition}  \label{prop.twistednoeth}
$C_{\geq 0}$ is noetherian.
\end{proposition}

We consider again our usual setting of a rank two two-sided vector space $V$ and the non-commuative symmetric algebra $A = \mathbb{S}^{nc}(V)$. Following Van den Bergh \cite[Section 6.3, Step 9]{vandenbergh}, we let $I \subset A$ be the ideal generated by the $g$. More precisely, for $i \in \mathbb{Z}$, we let $I_{i  i+2} \subset A_{i  i+2}$ denote the $K$-sub-bimodule generated by $g_{i}$.  For $j \geq i+2$, $I_{ij}$ is the image of $A_{i j-2}\otimes I_{j-2 j}$ in $A_{i j}$ under multiplication by Proposition \ref{prop.normalfamily}.

We come to our first major result.

\begin{theorem}  \label{thm.Anoeth}
$A/I$ is a twisted ring on $F$ if $V$ is simple and a twisted ring on $K$, if $V$ is non-simple. Furthermore, $A$ is noetherian.
\end{theorem}
The proof divides into the two cases and will be completed in the following two subsections.

\subsection{Proof case 1: $V$ simple} \label{section.twisted3}
We retain the notation from Section \ref{subsection:simplecase} so $V = \,_0F_1$. We wish to show that the $g_i$ are non-zero-divisors of $A$ and compute the quotient $A/I$. We first note the following.
\begin{lemma}  \label{lemma.bimodulestr}
Let $j>i$. The natural $F$-bimodule structure on $T_{ij}$ induces an $F$-bimodule structure on $(A/I)_{ij}$.
\end{lemma}
\begin{proof}
Let $J$ be the ideal of $T$ corresponding to $I$ so $T/J\simeq A/I$. Note that $J_{ii+2}$ is an $F$-bimodule by  Lemma~\ref{lemma.Fcentral}. Now $J$ is the ideal of $T$ generated by the $J_{ii+2}$ so we are done.
\end{proof}

\begin{lemma} \label{lemma.AmodI}
For $i< j$, $F \otimes 1 \otimes \ldots \otimes 1$ is a left $K_i$-space complement to $I_{ij}$ in $A_{ij}$. In particular, the left and right $F$-dimension of $(A/I)_{ij}$ is 1.
\end{lemma}
\begin{proof}
Note that $I_{ij} \neq A_{ij}$ since the left $K_i$-dimension of $I_{ij}$ is at most that of $A_{ij-2}$. Lemma~\ref{lemma.modFQ} and induction shows that every element of $A_{ij}$ is congruent modulo $I$ to one of the form $a \otimes 1 \otimes \ldots \otimes 1$. Hence the natural map
$$ F \otimes 1 \otimes \ldots \otimes 1 \rightarrow (A/I)_{ij} $$
is surjective. It must be injective as it is left $F$-linear and the left hand side has $F$-dimension 1 while the right hand side has dimension at least one. The lemma now follows.
\end{proof}
This lemma and Lemma~\ref{lemma.modFQ} completely describes the quotient algebra $A/I$ for given positive degree elements $a \otimes 1 \otimes \ldots \otimes 1 \in A_{ij}, b \otimes 1 \otimes \ldots \otimes 1 \in A_{jl}$, their product is congruent modulo $I$ to $a\sigma^{j-i}(b) \otimes 1 \otimes \ldots \otimes 1$ where $\sigma^{j-i}$ is the alternating product of $j-i$ $\sigma_0$'s and $\sigma_1$'s, the right most one being $\sigma_0$ if $j$ is even, and $\sigma_1$ if $j$ is odd. In other words, $A/I$ is a twisted ring on $F$ and it is in fact a domain whose ring of quotients is the corresponding full twisted ring defined above. This establishes the first assertion in Theorem~\ref{thm.Anoeth} when $V$ is simple.

Lemma~\ref{lemma.AmodI} also gives the following result.
\begin{proposition}  \label{prop:gisnzd}
Left and right multiplication by $g_i$ are monomorphisms.
\end{proposition}
\begin{proof}
We let $\dim$ denotes left $K_i$-dimension. Recall from Lemma~\ref{lemma.eulerses} that $\dim A_{ij} = j-i+1$. Then Lemma~\ref{lemma.AmodI} gives
$$\dim A_{ij}g_j = \dim A_{ij+2} - 2 = \dim A_{ij}$$
so right multiplication by $g_j$ must be injective. A symmetric argument gives the case of left multiplication.
\end{proof}


Note that $(A/I)_{ii} = K_i$ and $A/I$ equals the full twisted ring in positive degrees.
It follows from Proposition~\ref{prop.twistednoeth} that $A/I$ is noetherian. Now Proposition~\ref{prop:gisnzd} and the standard Hilbert basis argument show that $A$ is noetherian. This completes the proof of the theorem when $V$ is simple.

\subsection{Proof case 2: $V$ non-simple} \label{section.easytwist}
We retain the notation from Section \ref{section.oned}.  Thus $I_{i i+1}$ is the $K$-bimodule generated by $y$ for $i$ even and $x^{*}$ for $i$ odd.
Therefore, if $j > i$, then $(A/I)_{ij}$ is spanned by the image of a single simple tensor whose components consist of alternating $x$'s and $y^{*}$'s. In particular, the right and left $K$-dimension of $(A/I)_{ij}$ is $\leq 1$.  We next note that since the sequence
$$
A_{ij} \otimes I_{j j+1} \rightarrow A_{i j+1} \rightarrow (A/I)_{i j+1} \rightarrow 0
$$
whose left map is induced by multiplication is exact, the right and left dimension of $(A/I)_{i j+1}$ is $\geq 1$.  Therefore, it must equal one, which implies that the left map is monic.  It follows that left and right multiplication by $y$ or $x^{*}$ is monic, and that $A/I$ is a twisted ring on $K$. The twisting automorphisms are, in the notation of Lemma~\ref{lemma.twosidedclass}(1a) respectively (1b), $\tau$ or $\sigma^{-1}$ depending on parity respectively $\sigma$ or $\sigma^{-1}$ depending on parity. In fact, $A/I$ equals the full twisted ring in non-negative degrees so is noetherian by Proposition~\ref{prop.twistednoeth}. Once again, it follows from the Hilbert basis theorem that $A$ is noetherian too. The proof of Theorem~\ref{thm.Anoeth} is now complete.

\section{$A$ is a domain with global dimension 2}  \label{sec.domain}
In this section, we show that $A$ has global dimension two. This allows us to show that the Gelfand-Kirillov dimension is particularly nice and we have a theory of Hilbert polynomials. We then show that $A$ is a domain, generalizing \cite[Theorem 3.7]{nyman}.

Let $D$ be a $\mathbb{Z}$-indexed algebra. We say that $D$ is {\it connected} if $D_{ij} = 0$ whenever $i>j$ and all the $D_{ii}$ are fields. We assume further that $D$ is noetherian. In this section and the next, we let $\bar{D}$ be the $D$-bimodule $D/D_{\geq 1} = \oplus D_{ii}$. Many results which are true for connected graded noetherian algebras are true for connected indexed noetherian algebras too. We will not of course re-write the proofs of all these extensions, but hopefully demonstrate enough that the reader is aware of the main differences between the two theories, and can easily verify all unproved assertions by comparing with standard graded proofs.

The graded version of Nakayama's lemma holds for $D$-modules: if $M$ is a left bounded $D$-module and $M \uotimes \bar{D} = 0 $ then $M = 0$. This applies in particular to any noetherian module. By right exactness of $\uotimes$, if $M$ is noetherian and $\phi: P \ra M$ is a homomorphism such that $P \uotimes \bar{D} \ra M \uotimes \bar{D}$ is surjective, then $\phi$ is surjective too. We may thus talk about {\it minimal projective covers} and {\it minimal projective resolutions} of a noetherian module.

\begin{proposition}  \label{prop.gldim2}
The global dimension of $A$ is 2.
\end{proposition}
\begin{proof}
We show how the standard graded proof goes through. First note that a noetherian module $M$ is projective if and only if $\uTor^A_1(M, \bar{A}) =0$. Indeed if $\uTor^A_1(M, \bar{A}) =0$, we form the exact sequence
$$ 0 \ra N \ra P \xrightarrow{\phi} M \ra 0 $$
with $\phi$ a minimal projective cover. Tensoring with $\bar{A}$ shows that $N \uotimes \bar{A} = 0$ so $N$ is zero too.

Consider the projective resolution of the left $A$-module $A_{j+2,j+2}$ given by the Euler exact sequence (Lemma \ref{lemma.eulerses}(3)). Since the resolution has length 2, we see $\uTor^A_3(-,\bar{A}) = 0$ and given any noetherian right module $M$, the long exact sequence for Tor shows that its second syzygy must be projective. The proof for left modules uses the right hand version of \ref{lemma.eulerses}.
\end{proof}

By Proposition~\ref{prop.gldim2}, any noetherian object $M \in {\sf Gr }A$ has a finite resolution by a finite direct sum of $e_j A$, so exactness of the dimension function and Lemma~\ref{lemma.eulerses}(2) ensures that $f_M(n) :=\operatorname{dim}_{K_n}M_{n}$ is eventually of the form $cn+d$, with $c, d \in \mathbb{Z}$ and $c \geq 0$. If $c \neq 0$ then we say $c$ is the {\em multiplicity} of $M$, and otherwise $d$ is the {\em multiplicity} of $M$. Following usual conventions for defining the Gelfand-Kirillov dimension, we will look at the partial sums $\sum_{i=-\infty}^n f_M(i)$ which is also eventually a polynomial which we denote by $h_M(n)$. We define $\dim M = \deg h_M$. Of course, $\dim M = \deg f_M + 1$ so long as $M$ is not eventually zero. As usual, given an exact sequence of noetherian modules $0 \ra M' \ra M \ra M'' \ra 0$, we have $\dim M = \max\{\dim M', \dim M'' \}$.

We say that a noetherian module $M$ is {\em critical} (with respect to $\dim$) if for every non-zero submodule $N < M$, we have $\dim M/N < \dim M$. We say that it is {\em $d$-pure} (with respect to $\dim$) if $\dim M = d$ and $\dim N < d$ for all $N < M$.

We can now show $A$ is a domain.

\begin{theorem}  \label{thm.Adomain}
\begin{enumerate}
 \item $e_i A$ is 2-pure, critical and uniform.
 \item $A$ is a domain.
\end{enumerate}
In particular, by Proposition~\ref{prop.ringquotients}, $A$ has a ring of fractions.
\end{theorem}
\begin{proof}
 (Following \cite[Theorem~3.7]{nyman}.) Note $\dim e_i A = 2$ and let $J_i < e_i A$ be the sum of all submodules of dimension $\leq 1$. Since $\dim a_{ij} J_j \leq \dim J_j$ for any $a \in A_{ij}$, we have that $J := \oplus_i J_i$ is an ideal of $A$ with $J_{ij} = (J_i)_j$.

We first note that the conclusions of the theorem hold if $A$ is replaced with $A/J$. Indeed, $e_j (A/J)$ is 2-pure by definition of $J$ and exactness of $\dim$. Hence any non-zero submodule $N$ of $e_j(A/J)$ must be 2-dimensional with multiplicity at least 1. However, $e_j A$ has multiplicity 1 so the dimension of $(e_j\, A/J) / N$ must be at most 1. Hence $e_j (A/J)$ is critical. It is also uniform, for given two distinct non-zero submodules with zero intersection, their sum must be a submodule with multiplicity at least 2, a contradiction. We finally show that any non-zero $a_{ij} + J_{ij} \in A_{ij}/J_{ij}$ is a non-zero-divisor. Let $\phi: e_j (A/J) \ra e_i (A/J)$ be left multiplication by $a_{ij}$. It is non-zero as $a_{ij}e_j \notin J$ and thus injective since $e_j (A/J)$ is critical and $e_i (A/J)$ is pure.

It thus remains to show that $J = 0$. Unlike in \cite{nyman}, we will need to distinguish the left and right vector space dimensions of a $K_{i}-K_{j}$-bimodule, so we use the notation $l.\dim_K ,r.\dim_K$ for these respectively.

\begin{lemma}  \label{lemma.leftdimN}
 The left $A$-module $J e_i$ has dimension $\leq 1$.
\end{lemma}
\begin{proof}
We in fact show that the left dimensions $l.\dim_K J_{ij}$ are bounded. We first consider $J_0 = e_0 J$ which has dimension $\leq 1$ say with multiplicity $d$. We pick $j_0$ large enough so $r.\dim_K J_{0j} = d$ for all $j \geq j_0$. Let $g_j\in A_{j, j+\delta}$ be the normal element constructed in Section~\ref{sec.normal}. Then right multiplication by $g_j: J_{0j} \ra J_{0 j+\delta}$ is both left $K$-linear and, by normality, skew right $K$-linear. Since it is injective (Proposition~\ref{prop:gisnzd}) and the right dimensions are the same, it must be bijective. It follows that $l.\dim_K J_{0j} = l.\dim_K J_{0 j+\delta}$. Furthermore, we see that the left dimensions $l.\dim_K J_{0j}, j\geq j_0$ are bounded, say by $d'$. Increasing $j_0,d'$ if necessary, we may further assume that $l.\dim_K J_{1j} \leq d'$ for $j \geq j_0$.

We now show that $l.\dim_K J_{ij} \leq d'$ for all $i,j$. Let $A(2)$ denote the $\mathbb{Z}$-indexed algebra which is $A$ re-indexed so that $A(2)_{ij} = A_{i+2,j+2}$. Then the right double dual gives an isomorphism $A \simeq A(2)$. In particular, we see that $J_{ij} \simeq J_{i+2,j+2}$ as $K-K$-bimodules. Hence $l.\dim_K J_{ij} \leq d'$ so long as $j-i \geq j_0$. However, the non-zero elements of $A_{j-1,j}$ are non-zero-divisors by Lemma~\ref{lemma.Vnzd} so $l.\dim_K J_{ij-1} \leq l.\dim_K J_{ij}$ and we see $l.\dim_K J_{ij} \leq d'$ for all $i,j$.
\end{proof}
We now finish the proof of the theorem by showing $J_i = 0$. Let $L < Ae_i$ be its left annihilator. Let $a_1,\ldots, a_m$ be a finite set of homogeneous generators for $J_i$, $L_s$ be their left annihilators and $d_s$ be their degrees. Then we have the following inequality of left module dimensions
$$ \dim A e_i /L \leq \dim \oplus_{s=1}^m Ae_i/L_s  = \max_{s} \{\dim A a_s \} \leq \max_{s} \{\dim Je_{d_s}\} \leq 1.$$
Hence $\dim L = \dim A e_i = 2$ and by the lemma, we can find some $h\in \mathbb{Z}$ and $a_h \in L_h$ such that $a_h \notin J$. Now $A/J$ is a domain and left multiplication by $a_h$ annihilates $J_i$ so we have an exact sequence
$$ 0 \ra J_i \ra e_i A \xrightarrow{a_h} e_h A \ra e_h A / a_h e_i A \ra 0 .$$
It follows that $J_i$ is a second syzygy and since $\operatorname{gl.dim} A = 2$, it must be projective. However, the only projective module of dimension $\leq 1$ is 0.
\end{proof}

\section{Cohomology of $\mathbb{P}^{nc}(V)$ and bimodule species} \label{section.derived}

In this section, we introduce the non-commutative projective line $\mathbb{P}^{nc}(V)$, defined via Artin-Zhang's theory of non-commutative projective geometry \cite{az}. We then compute the cohomology of line bundles over $\mathbb{P}^{nc}(V)$. This will give us the classical Serre vanishing result, and allow us  to find a tilting bundle on $\mathbb{P}^{nc}(V)$. In particular, we obtain our desired derived equivalence between $\mathbb{P}^{nc}(V)$ and bimodule species. We will also show how $\mathbb{P}^{nc}(V)$ gives rise to the Dlab and Ringel's preprojective algebra $\Pi(V)$.

We now introduce notation that will be in effect for the remainder of this paper. Let $D$ be a $\mathbb{Z}$-indexed or graded algebra. We say a $D$-module is {\em graded torsion} if it is the direct limit of right bounded modules. Following Artin-Zhang we let ${\sf Proj} D$ denote the quotient category ${\sf Gr }D/{\sf Tors }D$ where ${\sf Tors }D$ denotes the full subcategory of ${\sf Gr }D$ consisting of graded torsion modules. In case $A = \mathbb{S}^{nc}(V)$ we write  $\mathbb{P}^{nc}(V)$ for ${\sf Proj}A$. Objects of $\mathbb{P}^{nc}(V)$ will often be referred to as {\em sheaves on } $\mathbb{P}^{nc}(V)$. The noetherian objects of $\mathbb{P}^{nc}(V)$ will be called {\em coherent sheaves} and their full subcategory is denoted ${\sf Coh }\mathbb{P}^{nc}(V)$.  We let $\pi:{\sf Gr }A \rightarrow \mathbb{P}^{nc}(V)$ denote the quotient functor, we let $\omega:\mathbb{P}^{nc}(V) \rightarrow {\sf Gr }A$ denote the section functor, which is the right adjoint of $\pi$, and we let $\tau:{\sf Gr }A \rightarrow {\sf Tors }A$ denote the functor sending a module to the sum of all its submodules which are objects of ${\sf Tors }A$.

As we show below, the computation of the cohomology of line bundles will follow from the formula for the right derived functors of $\tau$.  In \cite{duality}, this computation is carried out in case $V$ is algebraic and $K_{0}=K_{1}$.  Now suppose $V$ is non-algebraic or $K_{0} \neq K_{1}$.  In this case, one can check that the proofs and results in \cite{duality} through \cite[Corollary 4.12]{duality} still hold, as long as one modifies all statements to take into account that we may have $K_{0} \neq K_{1}$.  This allows us to import the computation of the right derived functors of $\tau$ to our setting, and to employ this result in order to compute the cohomology of line bundles.

To simplify the exposition in this and the following sections, we define
$\mathcal{A}_{i}:= \pi e_{i}A$. The $\mathcal{A}_i$ are the analogues of the line bundles $\mathcal{O}(-i)$ on $\mathbb{P}^1$ so we will refer to them as {\em line bundles} on $\mathbb{P}^{nc}(V)$ too.

In order to motivate our computation, recall that in the classical case $\mathbb{P}^1 = {\sf Proj }\, \mathbb{S}(K^2)$ where $\mathbb{S}(K^2)$ is the (commutative) symmetric algebra in two variables, we have
$$ \Ext^i_{\mathbb{P}^1}(\mathcal{O}(-m), \mathcal{O}(-n)) =
H^i(\mathbb{P}^1,\mathcal{O}(m-n)) =
\begin{cases}
\mathbb{S}^{m-n}(K^2) & \text{if } i=0 \\
\mathbb{S}^{n-2-m}(K^2)^* & \text{if } i=1 \\
0 & \text{else.}
\end{cases}
$$
We now proceed to show that we get a similar result in the non-commutative case.  By adjointness, we have the following isomorphism of functors from $\mathbb{P}^{nc}(V)$ to ${\sf Mod }K$:
\begin{eqnarray*}
\operatorname{Hom}_{\mathbb{P}^{nc}(V)}(\pi(e_{m}A),-) & \cong & \operatorname{Hom}_{{\sf Gr }A}(e_{m}A, \omega(-)) \\
& \cong & \operatorname{Hom}_{{\sf Mod }K}(A_{mm}, \omega(-)_{m}).
\end{eqnarray*}
It follows that
\begin{equation} \label{eqn.cohom}
\operatorname{Hom}_{\mathbb{P}^{nc}(V)}(\mathcal{A}_{m},-) \cong \omega(-)_{m}
\end{equation}
and thus
\begin{equation}   \label{eqn.homs}
\operatorname{Hom}_{\mathbb{P}^{nc}(V)}(\mathcal{A}_{m}, \mathcal{A}_{n}) \cong A_{nm}
\end{equation}
as right $K$-modules.  Furthermore, it is routine to check that the isomorphism (\ref{eqn.homs}) respects the natural left $K$-module structures.

Next, we compute $\operatorname{Ext}^{1}_{\mathbb{P}^{nc}(V)}(\mathcal{A}_{m},\mathcal{A}_{n})$.  We have
\begin{eqnarray} \label{eqn.linebundles1}
\operatorname{Ext}^{1}_{\mathbb{P}^{nc}(V)}(\mathcal{A}_{m},\mathcal{A}_{n}) & \cong & (R^{1}\omega(\mathcal{A}_{n}))_{m} \\
& \cong & (R^{2}\tau(e_{n}A))_{m} \label{eqn.linebundles2} \\
& \cong & A^{*}_{m,n-2}  \cong \ ^*A_{m+2,n}\label{eqn.linebundles3}
\end{eqnarray}
where the first isomorphism is from (\ref{eqn.cohom}), the second follows from \cite[Theorem 4.11]{duality}, and the third follows from \cite[Proposition 3.19 and Lemma 4.9]{duality}.

\begin{lemma}
The isomorphism $\operatorname{Ext}^{1}_{\mathbb{P}^{nc}(V)}(\mathcal{A}_{m},\mathcal{A}_{n}) \rightarrow A^{*}_{m,n-2}$ above is a bimodule isomorphism.
\end{lemma}

\begin{proof}
We need only check the isomorphism is compatible with the left $K$-module structure.  The fact that (\ref{eqn.linebundles1}) is compatible is routine and left to the reader.  Similarly, the proof that (\ref{eqn.linebundles3}) is compatible follows from the explicit description of the isomorphism from the proof of \cite[Theorem 4.4 and Lemma 4.9]{duality}, and we leave the straightforward details to the interested reader.  To prove that (\ref{eqn.linebundles2}) is compatible with left multiplication, we note that if $E^{\bullet}$ is an injective resolution of $e_{n}A$ in ${\sf Gr }A$, then each component is a direct sum of an injective graded torsion submodule and an injective graded torsion free submodule (so that $\tau$ applied to the module is zero) by \cite[Corollaire 2]{gab}.  Thus, as in the proof of \cite[Theorem 11.26]{smithnotes}, if $I^{\bullet}$ denotes the graded torsion subcomplex of $E^{\bullet}$, we get a short exact sequence of complexes of injective objects
$$
0 \rightarrow I^{\bullet} \rightarrow E^{\bullet} \rightarrow Q^{\bullet} \rightarrow 0,
$$
and the map (\ref{eqn.linebundles2}) is the connecting homomorphism
$$
R^{1}\omega(\mathcal{A}_{n}) \cong h^{1}(\omega \pi Q^{\bullet}) \cong h^{1}(Q^{\bullet}) \rightarrow h^{2}(I^{\bullet}) \cong R^{2}\tau(e_{n}A).
$$
Therefore, the fact that (\ref{eqn.linebundles2}) is compatible with multiplication follows from the functorality of the long exact sequence on cohomology.
\end{proof}

If $i>1$, a similar argument using \cite[Corollary 4.7]{duality} shows that
\begin{equation}  \label{eqn.cd1}
\operatorname{Ext}^{i}_{\mathbb{P}^{nc}(V)}(\mathcal{A}_{m},-) \cong (R^{i+1}\tau(-))_{m} = 0.
\end{equation}

Now we find a tilting object for $\mathbb{P}^{nc}(V)$.  To this end, we will need the right hand version of the Euler exact sequence from Lemma \ref{lemma.eulerses}(3) in $\mathbb{P}^{nc}(V)$:
\begin{equation}  \label{eqn.euler}
 0 \ra \mathcal{A}_{i+2} \ra (\mathcal{A}_{i+1})^{\oplus 2} \ra \mathcal{A}_i \ra 0.
\end{equation}
We also recall that the bimodule $V$ gives rise to a species which is defined to be the upper triangular matrix ring
\begin{equation}  \label{eqn.species}
\operatorname{Spe} V :=
\begin{pmatrix}
 K_{0} & V \\ 0 & K_{1}
\end{pmatrix}.
\end{equation}
It is clearly noetherian and hereditary.

\begin{prop} \label{prop.tilting}
For all $i \in \mathbb{Z}$, $\mathcal{T}_i := \mathcal{A}_{i} \oplus \mathcal{A}_{i+1}$ is a tilting object for $\mathbb{P}^{nc}(V)$. Furthermore, $\End_{\mathbb{P}^{nc}(V)} \mathcal{T}_i = \operatorname{Spe} V^{i*}$ so ${\sf Coh}\mathbb{P}^{nc}(V)$ and $\operatorname{Spe} V^{i*}$ are derived equivalent.
\end{prop}

\begin{proof}
We let $\mathcal{T}=\mathcal{A}_{i} \oplus \mathcal{A}_{i+1}$.  We first note that by (\ref{eqn.linebundles1}), (\ref{eqn.linebundles2}) and (\ref{eqn.linebundles3}), we have
$$
\operatorname{Ext}^{1}(\mathcal{A}_{i}\oplus \mathcal{A}_{i+1}, \mathcal{A}_{i}\oplus \mathcal{A}_{i+1})=0.
$$
The computation of the endomorphism ring of $\mathcal{T}_i$ is immediate from (\ref{eqn.homs}). Let ${\sf add T}$ be the full subcategory of $\mathbb{P}^{nc}(V)$ consisting of direct summands of direct sums of $\mathcal{T}_i$ and $\mathcal{C}$ be the smallest full subcategory containing ${\sf add T}$ which is closed under kernels of surjections. By (\ref{eqn.euler}), we see that $\mathcal{C}$ includes all the $\mathcal{A}_j$ for $j \geq i$. These form a set of generators for $\mathbb{P}^{nc}(V)$.  Finally, $\mathbb{P}^{nc}(V)$ has finite homological dimension by Proposition \ref{prop.gldim2} so we are done by \cite[Proposition~4.2]{keller}.
\end{proof}

Our next result is a version of Serre vanishing and Serre finiteness. Recall that $e_mA$ is a left $K_m$-module so $\Ext^i_{\mathbb{P}^{nc}(V)}(\mathcal{A}_m,\cM)$ is a right vector space over $K_m$.

\begin{thm}  \label{thm.serrevanish}
For any coherent sheaf $\cM \in {\sf Coh} \mathbb{P}^{nc}(V)$, the vector space
$$
\Ext^i_{\mathbb{P}^{nc}(V)}(\mathcal{A}_m,\cM)
$$
is finite dimensional over $K_m$. Furthermore, if $i>0$ then $\Ext^i_{\mathbb{P}^{nc}(V)}(\mathcal{A}_m,\cM) = 0$ for $m \gg 0$.
\end{thm}
\begin{proof}
For the convenience of the reader, we include the proof of this result, a more complicated version of which was proved in \cite[Theorem 3.5(2)]{finiteness}.
First note that the theorem holds for $\cM = \mathcal{A}_n$ by (\ref{eqn.cohom},\ref{eqn.linebundles3},\ref{eqn.cd1}). Since $A$ has global dimension 2 by Proposition~\ref{prop.gldim2}, an arbitrary coherent sheaf $\cM$ has a finite resolution by finite direct sums of line bundles $\mathcal{A}_n$. The result follows.
\end{proof}

Unlike the category of graded modules over a $\mathbb{Z}$-graded algebra, the category of graded modules over a $\mathbb{Z}$-indexed algebra $D$ does not have a natural shift functor. The analogue of the shift functor when it does exist is defined as follows.
Let $D(\delta)$ be the $\mathbb{Z}$-indexed algebra which is re-indexed so $D(\delta)_{ij} = D_{i+\delta,j+\delta}$. A {\it degree $\delta$-automorphism} of $D$ is an isomorphism $D \ra D(\delta)$.
Let $\phi:D \ra D(\delta)$ be a degree $\delta$-automorphism and $M$ be a right $D$-module.
\begin{defn}  \label{def.twist}
We define $M^{\phi}(\delta)$ to be the $D$-module whose degree $i$ component is $M^{\phi}(\delta)_i = M_{i+\delta}$ and whose multiplication map is
$$ M^{\phi}(\delta)_i \otimes D_{ij} \ra M^{\phi}(\delta)_j: m_{i+\delta} \otimes a_{ij} \mapsto m_{i+\delta} \phi (a_{ij}).$$
We call $M^{\phi}(\delta)$ the {\em twist of $M$ by $\phi$}.
\end{defn}
\noindent
Twisting is clearly an invertible functor, and it is straightforward to check that it induces an invertible functor on ${\sf Proj} D$.

We may finally relate the non-commutative symmetric algebra to Dlab and Ringel's {\em preprojective algebra} $\Pi(V)$ introduced in \cite{dlabringel2}. This is defined via generators and relations by
$$\Pi(V) = T(V \oplus V^*) /\mathcal{R}$$
and the ideal of relations $\mathcal{R}$ is generated by the images of the natural maps
$$
K_0 \ra V \otimes_{K_1} V^*
$$
and
$$
K_1 \ra V^* \otimes_{K_0} V
$$
defined in Section \ref{subsec.SncV}, where we have identified $V$ with $V^{**}$ using Lemma~\ref{lemma.ranktwo}. In the construction of the tensor algebra, we only allow alternating tensor products of $V$ and $V^*$ such as $V \otimes_{K_1} V^* \otimes_{K_0} V$ but not $V \otimes V$ which in general does not make sense. The preprojective algebra is often graded by putting $V, V^*$ in degree 1, though it will turn out below, that this is not the grading we want.

Before stating the next result, note that the isomorphism $V \ra V^{**}$ of Lemma~\ref{lemma.ranktwo} induces a degree two isomorphism $\phi: A \ra A(2)$.

\begin{proposition}  \label{prop.preprojective}
Let $\nu=(-)^{\phi}(2)$ be the twist by $\phi$ functor on $\mathbb{P}^{nc}(V)$ and $\mathcal{T}_{-1} = \mathcal{A}_0 \oplus \mathcal{A}_{-1}$. Then $\Pi(V)$ is isomorphic to the {\em orbit algebra}
$$\bigoplus_{n \geq 0}\Hom_{\mathbb{P}^{nc}(V)}(\mathcal{T}_{-1}, \nu^n\mathcal{T}_{-1}).$$
In particular we have ${\sf Proj}\, \mathbb{S}^{nc}(V) \cong {\sf Proj}\, \Pi(V)$.
\end{proposition}
\begin{proof}
Note that $\nu(e_iA)$ is a projective module generated by a single element in degree $i-2$. Hence $\nu\mathcal{A}_i \cong \mathcal{A}_{i-2}$. A simple computation using (\ref{eqn.homs}) shows that the orbit algebra is isomorphic (as an ungraded algebra) to $\Pi(V)$. The equivalence between the {\sf Proj} categories (when $\Pi(V)$ is endowed with a grading via the aforementioned isomorphism) now follows from \cite[Theorem~4.5]{az}. Indeed, Theorem~\ref{thm.serrevanish} shows that $\nu$ is ample and furthermore, $\Hom_{\mathbb{P}^{nc}(V)}(\mathcal{T}_{-1},\cM)$ is a finite $K_0 \times K_{-1}$-bimodule for any coherent sheaf $\cM$.
\end{proof}
This proposition shows that in some sense, $\Pi(V)$ is a non-commutative 2-Veronese of $\mathbb{S}^{nc}(V)$ so it is easy to recover the former from the latter.

\section{$A$ is Auslander regular}  \label{sec.Ausreg}

In this section, we prove that $A = \mathbb{S}^{nc}(V)$ is Auslander regular in an appropriate sense. This allows us to show that $\mathbb{P}^{nc}(V)$ is hereditary, generalizing a familiar result for $\mathbb{P}^1$. This important result will in turn be used in later sections to prove a type of Serre duality, and allows us to identify ${\sf D}^b({\sf Coh }\mathbb{P}^{nc}(V))$ with the repetitive category.

Given a $\mathbb{Z}$-indexed algebra $D$, a right (respectively left) $D$-module $M$, and $D$-bimodule $B$, we define
$$\uExt^i_D(M,B) = \oplus_i \Ext^i(M, e_i B)$$
which we note is naturally a left (respectively right) $D$-module. We similarly define $\uExt^i_D(B,M)$ and if $B'$ is another $D$-bimodule then we define the $D-D$-bimodule $\uExt^i(B,B') = \oplus_{i,j} \uExt^i (e_j B, e_i B')$. We define the {\it (traditional) grade} of $M$ to be
$$ j(M):= \min \{ i | \uExt^i_D(M,D) \neq 0\} $$
(it is more common now to replace the bimodule $D$ above with the balanced dualizing complex which usually shifts the values of the grade). We say that $D$ is {\it Auslander-Gorenstein} if $D$ has finite injective dimension and for any noetherian module $M$ (left or right) and submodule $N < \uExt^i_D(M,D)$, we have $j(N) \geq i$. If further $\operatorname{gl.dim} D$ is finite, we say $D$ is {\it Auslander regular}.

We will need the following terminology in the proof of the next result.  We say an $A$-module $M$ is {\em $g$-torsion} if each element of $M$ is annihilated by $g^n = \{g_ig_{i+\delta}\ldots g_{i+(n-1)\delta}\}$ for $n$ large enough. We say $M$ is {\em $g$-torsion free} if right multiplication by $g$ is injective on $M$.

\begin{proposition}  \label{prop.BisAG}
$B = A/I$ is Auslander-Gorenstein of injective dimension 1.
\end{proposition}
\begin{proof}
First note that the right hand version of the Euler exact sequence (Lemma \ref{lemma.eulerses}) shows that $\operatorname{id}_A e_jA = 2$. The standard graded proof that $\operatorname{id} B = 1$ holds as follows. Let $\delta = \deg g_i$. Note that left multiplication by $g_i$ induces the following exact sequence
$$ 0 \ra e_{i+\delta} A \xrightarrow{g_i} e_i A \ra e_i B \ra 0 .$$
Taking the direct sum over all $i$ gives the exact sequence
\begin{equation} \label{eqn.AmodIres}
 0 \ra A \xrightarrow{g}  A \ra  A/I \ra 0 .
\end{equation}
Consider a minimal injective resolution of right $A$-modules
$$ 0 \ra A \ra J^0 \ra J^1 \ra J^2 \ra 0 .$$
Now $A$ is $g$-torsion free so minimality ensures that $J^0$ must be $g$-torsion free too and hence  $\underline{\Hom}_A(A/I,J^0) = 0$. Adjointness of $\underline{\Hom}, \uotimes$ ensures that the $\underline{\Hom}_A(A/I,J^j)$ are injective $A/I$-modules so applying $\underline{\Hom}_A(A/I,-)$ to the injective resolution for $A$ gives the $A/I$-injective resolution
$$ 0 \ra \uExt^1_A(A/I,A) \ra \underline{\Hom}_A(A/I,J^1) \ra \underline{\Hom}_A(A/I,J^2) \ra 0 .$$
Now from the resolution (\ref{eqn.AmodIres}) for $A/I$ we see that $\uExt^1_A(A/I,A) \simeq A/I$. We thus have a length one $B$-injective resolution of $B$.

To check the Auslander condition for $B$, we use its ring of fractions $C$ as decribed in Section~\ref{section.twisted}. Now, the global dimension of $C$ is zero by Proposition~\ref{prop.structureQ}(2). Since $B$ has injective dimension 1, we need only check that for any noetherian $B$-module $M$ and submodule $N$ of $\uExt^1_B(M,B)$ we have $j(N) \neq 0$. For definiteness, we work with a right module $M$. Now localization is exact so $C \uotimes_B  \uExt^1_B(M,B) = \uExt^1_C(M \uotimes_B C , C) = 0$. Hence $C \uotimes_B N = 0$ so $N$ is torsion and we must have $\underline{\Hom}_B(N,B) = 0$. Thus $j(N) >0$ and the proposition is proved.
\end{proof}

Before proving that $A$ is Auslander regular, we need to recover some of the general theory of Auslander-Gorenstein rings in the $\mathbb{Z}$-indexed case.  It turns out that the proofs as written in \cite{yz} generalize better than the old original papers. For example, Levasseur in \cite{levasseur} sometimes passes to the ungraded algebra which makes no sense in the indexed setting. Hence we follow the treatment in \cite{yz} closely.

Suppose $D$ is an Auslander-Gorenstein algebra of injective dimension $d$. We define the {\it canonical dimension} of a $D$-module $M$ to be $\Cdim M = \Cdim_D M= d-j(M)$. If we always have $\Cdim M = \dim M$ then we will say that $D$ is {\em Macaulay}.
Let ${\sf D}_{fg}(D)$ denote the derived category of complexes of $D$-modules with finitely generated cohomologies. Then as in \cite[Proposition~1.3]{yz}, $\mathbf{R}\underline{\operatorname{Hom}}_D(-,D)$ induces a duality between ${\sf D}_{fg}(D)$ and ${\sf D}_{fg}(D^{op})$.
As already remarked in Section~\ref{sec.torsion}, this can be computed by replacing $D$ with a finite complex of bimodules which is simultaneously left and right injective. Hence we still have the usual double-Ext spectral sequence \cite[Proposition~1.7]{yz} in this setting so \cite[Lemma~2.11]{yz} holds to give the following exactness result for $\Cdim$.

\begin{proposition}  \label{prop.cdimexact}
Given a short exact sequence $0 \ra M' \ra M \ra M'' \ra 0$ of noetherian $D$-modules we have
$$ \Cdim M = \max \{ \Cdim M', \Cdim M'' \} .$$
\end{proposition}

One may now repeat the proofs in \cite[Proposition~2.14-Corollary~2.17]{yz} and \cite[Lemma~1.1-Corollary~1.3]{asz} to obtain the following results. Below we let $\operatorname{Kdim}$ denote the {\em Krull dimension} as defined in standard ring theory texts such as \cite[Chapter~6]{mcr}.

\begin{thm}  \label{thm.onCdim}
Let $M$ be a noetherian $D$-module.
\begin{enumerate}
\item $\operatorname{Kdim} M \leq \Cdim M$.
\item $M$ has a $\Cdim$-critical composition series.
\end{enumerate}
\end{thm}

\begin{corollary}  \label{cor.AmodImacaulay}
$A/I$ is Macaulay.
\end{corollary}
\begin{proof}
Let $M$ be a noetherian $A/I$-module. If $\Cdim M = 0$ then Theorem~\ref{thm.onCdim}(1) shows that $M$ has finite length so $\dim M = 0$. If on the other hand $\Cdim M =1$ then $\underline{\Hom}_{A/I}(M,A/I) \neq 0$ so M has a non-zero homomorphic image in some $e_i A/I$. But from the twisted ring structure of $A/I$ (Theorem~\ref{thm.Anoeth}), we see that $e_i A/I$ is 1-pure so $\dim M =1$.
\end{proof}

We have not yet studied the notion of balanced dualizing complexes for even in the graded or local case, one usually assumes that $\bar{D}$ is finite dimensional over some central base field. Nevertheless, some of the results of \cite[Section~4]{yz} hold.

\begin{proposition}  \label{prop.AmodIn}
Let $n \in \mathbb{N}$.
\begin{enumerate}
\item For any noetherian $A/I^n$-module $M$, we have $\Cdim_{A/I^n} M = \Cdim_A M$.
\item $A/I^n$ is Auslander Gorenstein.
\item $A/I^n$ is Macaulay.
\end{enumerate}
\end{proposition}
\begin{proof}
We follow the treatment of \cite[Corollary~4.15, Proposition~4.16]{yz}. To prove (1), note first that $A/I^n$ has injective dimension 1 as a module over itself by the argument in the proof of Proposition \ref{prop.BisAG}. It thus suffices to show that $j_{A/I^n} M  = j_A M - 1$. Consider the change of rings spectral sequence
$$ \uExt^p_{A/I^n}(-, \uExt^q_A(A/I^n,A)) \Longrightarrow \uExt^{p+q}_A(-,A).$$
Now $I^n$ is generated by the normal family $g^n = \{g_ig_{i+\delta}\ldots g_{i+(n-1)\delta}\}$ so we may compute $\uExt^q_A(A/I^n,A)$ and see the spectral sequence collapses to
\begin{equation} \label{eqn.doubleextcollapse}
\uExt^p_{A/I^n}(-,A/I^n) \simeq \uExt^{p+1}_A(-,A).
\end{equation}
Part (1) follows.

To prove part (2), we use induction on $n$, the case $n=1$ being Proposition~\ref{prop.BisAG}. Suppose $A/I^{n-1}$ is Auslander-Gorenstein and let $M$ be a noetherian $A/I^n$-module. Consider the exact sequence
\begin{equation} \label{eqn.torsionSES}
0 \ra MI \ra M \ra M/MI \ra 0 .
\end{equation}
Note that $MI,M/MI$ are $A/I^{n-1}$-modules and by part (1), their $j$ numbers as $A/I^n$-modules and $A/I^{n-1}$-modules are the same. Consider the long exact sequence
$$ \ldots \ra \uExt^p_{A/I^n}(M/MI, A/I^n) \ra \uExt^p_{A/I^n}(M, A/I^n) \ra \uExt^p_{A/I^n}(MI, A/I^n) \ra \ldots .$$
This shows that any submodule $N$ of $\uExt^p_{A/I^n}(M, A/I^n)$ sits in a short exact sequence of the form
$$ 0 \ra N' \ra N \ra N'' \ra 0$$
where $N', N''$ are subquotients of $\uExt^p_{A/I^n}(M/MI, A/I^n),\uExt^p_{A/I^n}(MI, A/I^n)$ respectively so have $j$ number $\geq p$ by induction. It follows that $j(N) \geq p$ too and we are done.

(3) Now follows from induction and exactness of $\dim$ and $\Cdim$ applied to (\ref{eqn.torsionSES}).
\end{proof}

If $D$ is a connected graded noetherian algebra with a normal element $g$ of positive degree, then it is a classical result of Levasseur \cite{levasseur}, that $D$ is Auslander-Gorenstein if and only if $D/(g)$ is. To show $A$ is Auslander regular, we need a similar result obtained by mimicking the proof in \cite[Section~5]{yz}.

\begin{theorem}  \label{thm.AisAG}
 $A$ is Auslander regular and Macaulay.
\end{theorem}
\begin{proof} The proof in \cite[Theorem~5.1]{yz} can be carried over to our setting and we only indicate how to adapt their graded proof to our indexed setting. The proof essentially divides up into two cases, when a noetherian $A$-module  $M$ is $g$-torsion, and when it is $g$-torsion free. The first case is disposed of using Corollary~\ref{cor.AmodImacaulay} and Proposition~\ref{prop.AmodIn}.

To dispose of the $g$-torsion free case we  need to introduce some new notation. Since $g=\{g_i\}$ is a normal family of degree $\delta$ with all $g_i$ non-zero-divisors, {\em conjugation} by the $g_i$ is a degree $\delta$-automorphism of $A$. Explicitly, this conjugation map is
$$\phi: D_{ij} \ra D_{i+\delta,j+\delta}: a_{ij} \mapsto g_i^{-1} a_{ij} g_j.$$
For any right $D$-module $M$, the multiplication by $g_i$ maps $g_i: M_i \ra M_{i+\delta}$ assemble to give a $D$-module morphism
\begin{equation} \label{eqn.gmult}
M^{\phi^{-1}}(-\delta) \xrightarrow{g} M 
\end{equation}
which is injective if $M$ is $g$-torsion-free. Analyzing the exact sequence
$$0 \ra M^{\phi^{-1}}(-\delta) \xrightarrow{g} M  \ra M / M (g_i) \ra 0$$
as in \cite[Theorem~5.1]{yz} completes the proof.
\end{proof}

\begin{corollary}  \label{cor.projAhereditary}
$\mathbb{P}^{nc}(V)$ is hereditary.
\end{corollary}
\begin{proof}
Let $M$ be any noetherian $A$-module and
$$0 \ra M \ra J^0 \ra J^1 \ra J^2 \ra 0$$
be a minimal injective resolution. It suffices to show that $J^2$ is graded torsion for then as in \cite[(7.1.4)]{az}, it will induce a length one injective resolution in $\mathbb{P}^{nc}(V)$. The proof of \cite[Proposition~2.4]{asz} can be repeated to show that $J^2$ is essentially 0-pure with respect to the canonical dimension. Note that by Theorem~\ref{thm.onCdim}, a noetherian module has canonical dimension 0 if and only if it has finite length. It now follows from \cite[Proposition~2.2(2)]{az} that $J^2$ is itself graded torsion.
\end{proof}

The Auslander-Gorenstein theory proof of \cite[Proposition~2.5(3)]{asz} can be repeated in the $\mathbb{Z}$-indexed context to reprove the following result of Nyman's \cite[Lemma~3.5]{nyman}.

\begin{lemma}  \label{lemma.pd2}
If $M$ is a noetherian $A$-module, then $\underline{\Hom}_A(\bar{A}, M) \neq 0$ iff $\operatorname{pd}_A M = 2$.
\end{lemma}

\section{Classical Serre duality and Hilbert functions}  \label{sec.Serre}

If $V$ is not algebraic, then $\mathbb{P}^{nc}(V)$ is not Hom-finite, so there is no chance of defining a Serre functor. However, classical versions of Serre duality do hold. We examine such results in this section and apply them to the study of Hilbert functions of torsion sheaves.

A coherent sheaf is {\em torsion} (respectively {\em torsion free}) if it has the form $\pi M$ for some torsion (respectively torsion free) $A$-module $M$.

Note that $\operatorname{End} \mathcal{A}_i = K_{i}$ so $\Ext^p_{\mathbb{P}^{nc}(V)}(\mathcal{A}_i,\mathcal{M})$ is naturally a right vector space over $K$ whilst $\Ext^p_{\mathbb{P}^{nc}(V)}(\mathcal{M},\mathcal{A}_i)$ is a left vector space. When $A$ is the usual symmetric algeba $S^*(K^2)$, then $\Ext^p_{\mathbb{P}^1}(\mathcal{A}_0,-) = H^p(\mathbb{P}^1,-)$ and $\omega_{\mathbb{P}^1} \cong \mathcal{A}_2$ so the following generalizes Serre's original duality theorem.

\begin{theorem}  \label{thm.serre}
For any coherent sheaf $\mathcal{M} \in {\sf Coh } \mathbb{P}^{nc}(V)$, there is a natural isomorphism of right $K$-spaces
$$\Ext^{1-p}_{\mathbb{P}^{nc}(V)}(\mathcal{A}_i,\mathcal{M}) \simeq \ ^*\Ext^p_{\mathbb{P}^{nc}(V)}(\mathcal{M}, \mathcal{A}_{i+2}).$$
\end{theorem}
\begin{proof}
This follows from the usual proof for classical Serre duality \cite{harts} and the cohomology of line bundles computations in Section~\ref{section.derived}. The reader can check \cite{chan} for details.
\end{proof}

 To study Hilbert functions, we need to recall the theory of Euler characteristics. For arbitrary coherent $\cM$, we define the function
$$\chi(\cM,n) = \dim_K \Hom_{\mathbb{P}^{nc}(V)}(\mathcal{A}_n, \cM) - \dim_K \Ext^1_{\mathbb{P}^{nc}(V)}(\mathcal{A}_n, \cM).$$
Now $\omega \mathcal{M} = \oplus_n \Hom_{\mathbb{P}^{nc}(V)}(\mathcal{A}_n, \cM)$ so there is a strong relationship between the function $\chi(\cM,-)$ and the Hilbert function $f_{\omega \cM}$ introduced in Section~\ref{sec.domain}. Indeed, (\ref{eqn.homs}, \ref{eqn.linebundles3}) show that
$$\chi(\mathcal{A}_i,n) = n-i+ 1, \text{ for all } n .$$
Since every coherent sheaf has a resolution by direct sums of $\mathcal{A}_i$, we see that $\chi(\cM,n)$ is a polynomial which equals $f_{\omega \cM}(n)$ for all sufficiently large $n$.

\begin{corollary}  \label{cor.hilbconstant}
Let $\cM$ be a torsion coherent sheaf on $\mathbb{P}^{nc}(V)$. Then $f_{\omega\cM}$ is constant.
\end{corollary}
\begin{proof}
Note that $\cM = \pi M$ for some torsion module $M$. Since it is torsion, $j(M) >0$ and $\dim M = \Cdim M \leq 1$ as $A$ is Macaulay. It follows that $\chi(\cM,n)$ is constant. From classical Serre duality Theorem~\ref{thm.serre}, we know that $\Ext^1_{\mathbb{P}^{nc}(V)}(\mathcal{A}_n, \cM)=0$ so $f_{\omega\cM} = \chi(\cM,n)$ is constant too.
\end{proof}

Before proving our second version of Serre duality, we need to introduce some notation. Let $g = \{g_i\}$ be the normal family of elements in $A= \mathbb{S}^{nc}(V)$ defined in Section~\ref{sec.normal} and $\phi:A \rightarrow A(\delta)$ be conjugation by $g$ where $\delta = \deg g$. Note that $\phi$ is left and right $K$-linear in the sense that $K_i = K_{i+\delta}$ and with this identification, the restricted maps $\phi:A_{ij} \rightarrow A_{i + \delta,j + \delta}$ are $K_i-K_j$-bimodule maps. Also $\phi(g) = g$ as in the usual graded case. Analogous to (\ref{eqn.gmult}), right multiplication by $g$ induces a morphism
$$\cM^{\phi^{-1}}(-\delta) \xrightarrow{g} \cM.$$
We denote the image of this map by $\cM g$.

\begin{lemma}  \label{lem.gmap}
 Let $\cM, \mathcal{N}$ be sheaves on $\mathbb{P}^{nc}(V)$. The maps
$$(g,\mathcal{N}): \Ext^i_{\mathbb{P}^{nc}(V)}(\cM, \mathcal{N}) \ra \Ext^i_{\mathbb{P}^{nc}(V)}(\cM^{\phi^{-1}}(-\delta), \mathcal{N}), $$
$$(\mathcal{M},g): \Ext^i_{\mathbb{P}^{nc}(V)}(\cM, \mathcal{N}) \ra \Ext^i_{\mathbb{P}^{nc}(V)}(\cM, \mathcal{N}^{\phi}(\delta))$$
have isomorphic kernels and cokernels.
\end{lemma}
\begin{proof}
 The case $i=0$ follows from the commutative diagram
$$\begin{CD}
   \Hom(\cM,\mathcal{N}) @= \Hom(\cM,\mathcal{N}) \\
 @V{(g,\mathcal{N})}VV @VV{(\cM,g)}V  \\
 \Hom(\cM^{\phi^{-1}}(-\delta),\mathcal{N}) @>{(-)^{\phi}(\delta)}>> \Hom(\cM,\mathcal{N}^{\phi}(\delta))
  \end{CD}$$
The general case follows from taking an injective resolution $\mathcal{I}^{\bullet}$ for $\mathcal{N}$ and using the injective resolution $\mathcal{I}^{\bullet \phi}(\delta)$ for $\mathcal{N}^{\phi}(\delta)$.
\end{proof}

Note that $e_iA/g_iA$ is also a left vector space over $K_i$. Our final version of Serre duality is

\begin{proposition}  \label{prop.SerreAmodg}
 For any coherent sheaf $\cM$, there is a natural isomorphism of right $K$-spaces
$$\Hom_{\mathbb{P}^{nc}(V)}(\mathcal{A}_i/\mathcal{A}_ig, \cM) \cong \ ^*\Ext^1_{\mathbb{P}^{nc}(V)}(\cM, \mathcal{A}_{i+2}/\mathcal{A}_{i+2}g) .$$
\end{proposition}
\begin{proof}
Below we drop the subscript $\mathbb{P}^{nc}(V)$ for Hom and Ext spaces and use classical Serre duality (Theorem~\ref{thm.serre}) and Lemma~\ref{lem.gmap}.
\begin{eqnarray*}
 \Hom(\mathcal{A}_i/\mathcal{A}_ig, \cM) & \cong & \ker \left( (g,\cM): \Hom(\mathcal{A}_i,\cM) \ra \Hom(\mathcal{A}_i^{\phi^{-1}}(-\delta),\cM)\right) \\
 & \cong & \ker \left( (\mathcal{A}_i,g): \Hom(\mathcal{A}_i,\cM) \ra \Hom(\mathcal{A}_i,\cM^{\phi}(\delta))\right) \\
 & \cong & \ker \left( (g,\mathcal{A}_{i+2}): \ ^*\Ext^1(\cM, \mathcal{A}_{i+2}) \ra \ ^*\Ext^1(\cM^{\phi}(\delta), \mathcal{A}_{i+2})\right)\\
 & \cong & \ker \left( (\cM^{\phi}(\delta),g): \ ^*\Ext^1(\cM^{\phi}(\delta), \mathcal{A}_{i+2}^{\phi}(\delta)) \ra \ ^*\Ext^1(\cM^{\phi}(\delta), \mathcal{A}_{i+2})\right)\\
 & \cong & \ ^*\Ext^1(\cM^{\phi}(\delta), \mathcal{A}_{i+2}^{\phi}(\delta)/\mathcal{A}_{i+2}^{\phi}(\delta)g) \\
 & \cong & \ ^*\Ext^1_{\mathbb{P}^{nc}(V)}(\cM, \mathcal{A}_{i+2}/\mathcal{A}_{i+2}g).
\end{eqnarray*}

\end{proof}

\section{Sheaves on $\mathbb{P}^{nc}(V)$ and modules over corresponding species}  \label{sec.classify}

In this section we give a fairly complete description of coherent sheaves on $\mathbb{P}^{nc}(V)$ and then use the derived equivalence from Proposition \ref{prop.tilting} to give the analogous description of modules over the species $\Lambda = \operatorname{Spe} V$ defined in (\ref{eqn.species}).

We fix the tilting bundle $\mathcal{T} = \mathcal{A}_0 \oplus \mathcal{A}_1$ and let $\Phi= \operatorname{RHom}_{\mathbb{P}^{nc}(V)}(\mathcal{T},-): {\sf D}^b(\mathbb{P}^{nc}(V)) \ra {\sf D}^b(\Lambda)$. Recall that, by Corollary \ref{cor.projAhereditary}, $\mathbb{P}^{nc}(V)$ is hereditary so every bounded complex of sheaves is isomorphic to a direct sum of shifts of sheaves. There is a similar result for complexes of $\Lambda$-modules. In particular, $\Phi$ gives a bijection between indecomposable coherent sheaves and indecomposable modules. More precisely, given an indecomposable coherent sheaf $\mathcal{M}$, there is a unique integer $n$ such that the shift $\Phi(\mathcal{M})[n]$ is an indecomposable $\Lambda$-module and every indecomposable module arises this way.

\subsection{Classification of torsion-free sheaves} \label{subsec.class}

We wish to recover the following results of \cite[Theorem~3.14]{nyman}:
 a non-commutative Grothendieck's splitting theorem for torsion-free sheaves and the fact that every coherent sheaf is the direct sum of a torsion sheaf with a torsion free sheaf. Our treatment here is parallel to that in \cite{nyman} but bypasses the general torsion theory for non-commutative integral spaces by using the torsion theory for $\mathbb{Z}$-indexed algebras developed in Section~\ref{sec.torsion}.

We start with Grothendieck splitting.

\begin{theorem}  \label{thm.Grothendieck}
Any torsion-free coherent sheaf on $\mathbb{P}^{nc}(V)$ is isomorphic to a direct sum of $\mathcal{A}_i$'s (for possibly different $i$).
\end{theorem}
\begin{proof}
Let $M$ be a torsion-free noetherian $A$-module. We wish to show $\pi M$ is a direct sum of $\mathcal{A}_i$'s. By Proposition~\ref{prop.structuretorfree}, one can construct an exact sequence of the form
$$ 0 \ra M \ra (e_{i}A)^{\oplus n} \xrightarrow{\phi} N \ra 0 $$
for some module $N$. We may assume that the graded torsion submodule $\tau N$ of $N$ is zero, for otherwise we may replace $M$ with $M':=\phi^{-1} \tau N$ and note that $\pi M' \simeq \pi M$. It follows now from Lemma~\ref{lemma.pd2} that $\operatorname{pd}_A N \leq 1$ so $M$ must be projective. Hence $M$ is a direct sum of $e_i A$ and the theorem is proved.
\end{proof}

Let $M$ be a noetherian $A$-module and $\mathcal{M} = \pi M$. We let $M'$ be the torsion submodule and define $\mathcal{M}' := \pi M'$ to be the {\em torsion subsheaf} of $\mathcal{M}$. We consider the exact sequence
$$ 0 \ra \mathcal{M}' \ra \mathcal{M} \ra \mathcal{M}/\mathcal{M}' \ra 0 .$$
Now $\mathcal{M}/\mathcal{M}'$ is torsion-free and hence a direct sum of $\mathcal{A}_i$'s. By Serre duality (Theorem \ref{thm.serre}), we know that
$$\Ext^1_{\mathbb{P}^{nc}(V)}(\mathcal{A}_i,\mathcal{M}')^* \simeq \Hom_{\mathbb{P}^{nc}(V)}(\mathcal{M}', \mathcal{A}_{i+2})=0$$
so the exact sequence splits giving the next result which was also proved in \cite[Theorem~3.14]{nyman}:
\begin{proposition}  \label{prop.torsionsplitsoff}
Any coherent sheaf on $\mathbb{P}^{nc}(V)$ is a direct sum of its torsion subsheaf and a torsion-free sheaf.
\end{proposition}

\subsection{Classification of irregular indecomposable modules over the species}

We are now in a position to give a complete description of irregular $\Lambda$-modules using the theory of Section \ref{subsec.class}.
Let
$e_0 =\left(
\begin{smallmatrix}
1 & 0 \\ 0 & 0
\end{smallmatrix}\right),
e_1 =\left(
\begin{smallmatrix}
0 & 0 \\ 0 & 1
\end{smallmatrix}\right)
$ be the standard diagonal idempotents in $\Lambda$. Following Ringel \cite[Section~7]{species}, we define the {\em defect} of a right $\Lambda$-module $M$ to be
$$\partial M = \dim_K M e_0 - \dim_K M e_1 .$$
A $\Lambda$-module $M$ is said to be {\em regular}, if it is a direct sum of indecomposable modules of defect 0. A module is otherwise said to be {\em irregular}. We sometimes use the notation $M_0 \otimes_K V \ra M_1$ to denote the $\Lambda$-module it defines.

We now show that the $\mathcal{A}_i$ give irregular modules. Indeed, if $i \leq 1$, then (\ref{eqn.homs}) shows that the corresponding $\Lambda$-module is
$$\Phi(\mathcal{A}_i) : A_{i0} \otimes_K V \ra A_{i1} .$$
Here the multiplication map is just multiplication in the $\mathbb{Z}$-indexed algebra and the defect is $-1$. On the other hand, if $i > 1$, then (\ref{eqn.linebundles1}), (\ref{eqn.linebundles2}), and (\ref{eqn.linebundles3}) gives the corresponding $\Lambda$-module as
$$\Phi(\mathcal{A}_i)[1]: A_{0i-2}^* \otimes_K V \xrightarrow{m} A_{1i-2}^*.$$
Here the defect is $1$ and the multiplication map $m$ is given by dualizing the multiplication in $A$ map $V \otimes_K A_{1i-2} \ra A_{0i-2}$.

We next examine the defect of $\Phi(\mathcal{M})$ for a torsion sheaf $\mathcal{M}$.

\begin{proposition}  \label{prop.defect0}
Let $\cM$ be a torsion coherent sheaf on $\mathbb{P}^{nc}(V)$. Then $\Phi(\cM)$ has defect zero.
\end{proposition}
\begin{proof}
This follows from Corollary~\ref{cor.hilbconstant} since
\begin{eqnarray*}
\dim_K \Phi(\cM)e_0 & = & \dim_K \Hom_{\mathbb{P}^{nc}(V)}(\mathcal{A}_0,\cM) \\
& = & \dim_K \Hom_{\mathbb{P}^{nc}(V)}(\mathcal{A}_1,\cM) \\
& = & \dim_K \Phi(\cM)e_1,
\end{eqnarray*}
\end{proof}

Putting our computations together with our structure theory for coherent sheaves Theorem~\ref{thm.Grothendieck}, Proposition~\ref{prop.torsionsplitsoff} immediately gives the next result.
\begin{corollary}  \label{cor.regularistorsion}
The full subcategory of regular $\Lambda$-modules corresponds via $\Phi$ to the full subcategory of coherent torsion sheaves on $\mathbb{P}^{nc}(V)$.
\end{corollary}

We also obtain the following result of Ringel \cite[6.6 Lemma]{species}. The only difference is that his description of irregular modules are given in terms of the Coxeter functor, whereas ours are immediately read off the non-commutative symmetric algebra. A similar description is also given in \cite{dlabringel2} except that they use the preprojective algebra as opposed to our non-commutative symmetric algebra.

\begin{proposition}  \label{prop.classifyirregular}
Let $M$ be an indecomposable irregular $\Lambda$-module. Then it is isomorphic to one of the following:
$$ A_{i0} \otimes_K V \ra A_{i1} \quad \text{or} \quad
A_{0i-2}^* \otimes_K V \ra A_{1i-2}^*.$$
Any such module is uniquely determined up to isomorphism by the pair of integers $(\dim_K Me_0, \dim_K Me_1)$ which must differ by 1.
\end{proposition}

\subsection{Classification of $g$-torsion sheaves}

In this section we classify the following torsion coherent sheaves. Let $g = \{g_i\}$ be the normal family of elements defined in Section~\ref{sec.normal} and $I$ be the ideal generated by $g$. Recall a noetherian $A$-module $M$ is said to be {\em $g$-torsion} if each element is annihilated by $g^{n+1} = g_ig_{i+\delta}\ldots g_{i+n\delta}$ for sufficiently large $n$. In this case, the corresponding sheaf $\pi M$ will be said to be {\em $g$-torsion} too. Let ${\sf T}$ be the full subcategory of ${\sf Coh}\,\mathbb{P}^{nc}(V)$ consisting of $g$-torsion sheaves. It is straightforward to check that ${\sf T}$ is abelian. Our goal is to show that ${\sf T}$ is uniserial.

We omit the elementary proof of the following lemma, which we will use repeatedly without mention in the future.

\begin{lemma}
A coherent sheaf $\cM$ is $g$-torsion if and only if $\cM g^n = 0$ for some $n\geq 0$.
\end{lemma}

Let $B = A/I$ be the twisted ring studied in Section~\ref{section.twisted}. We saw that its ring of fractions $C$ is in fact the full twisted ring. In fact, the embedding $B \hookrightarrow C$ is an isomorphism in positive degrees. Furthermore, if $\delta = \deg g$, then $\dim_K C_{ij} = \delta$ both as a left and right vector space. We wish first to study the $g$-torsion sheaf $\mathcal{B}_i := \pi(e_iB) = \mathcal{A}_i/\mathcal{A}_ig$. The first result about this sheaf is the following.

\begin{proposition}  \label{prop.Bisimple}
$\omega \mathcal{B}_i = e_i C$. In particular, $\mathcal{B}_i$ is simple.
\end{proposition}
\begin{proof}
From the structure of $C$ as a full twisted ring, we see that $\underline{\Hom}_A(\bar{A}, e_iC) = 0$ so $e_iC$ embeds in $\omega \mathcal{B}_i$. Equality follows by comparing Hilbert functions using the fact that $f_{\omega \mathcal{B}_i}$ is constant by Corollary~\ref{cor.hilbconstant}.

To prove simplicity when $V$ is non-simple, note that $\delta =1$ so the Hilbert function is the constant 1 and we are done. When $V$ is simple, then from the multiplication formula (\ref{eqn.muC}) for $C$, we see that any non-zero element of $e_iC$ generates $e_iC$ in sufficiently high degree. Hence $\mathcal{B}_i$ is simple.
\end{proof}

This easily gives the endomorphism rings of the $\mathcal{B}_i$.

\begin{cor}  \label{cor.endBi}
If $\delta = 1$ then $\End \mathcal{B}_i = K_i$. If $\delta = 2$ then $\End \mathcal{B}_i = F$ where $F$ is the field defined in Lemma~\ref{lemma.twosidedclass}(2).
\end{cor}
\begin{proof}
First note the following inequality of right vector space dimensions
$$\dim_K \End\, \mathcal{B}_i \leq \dim_K \Hom(\mathcal{A}_i, \mathcal{B}_i) = \delta.$$
When $\delta = 1$, then left multiplication by $K$ on $e_i B$ induces an embedding $K \hookrightarrow \End\,\mathcal{B}_i$ which is an isomorphism by the inequality above. When $\delta = 2$, then left multiplication by $F$ on $e_i C$ gives an embedding $F \hookrightarrow \End\, \mathcal{B}_i$ which again is an isomorphism as before.
\end{proof}

\begin{proposition}  \label{prop.Biuniquesimple}
Let $i \in \mathbb{Z}$.
\begin{enumerate}
\item Let $\cM\neq 0$ be a coherent sheaf with $\cM g =0$. Then there is a non-zero embedding $\mathcal{B}_i \hookrightarrow \cM$.
\item $\mathcal{B}_i \cong \mathcal{B}_j$ for any $j \in \mathbb{Z}$.
\item In particular, $\mathcal{B}_0$ is the unique simple $g$-torsion sheaf.
\end{enumerate}
\end{proposition}
\begin{proof}
We first prove part (1). Corollary~\ref{cor.hilbconstant} shows that the Hilbert function $f_{\omega \cM}$ is constant non-zero so there is a non-zero homomorphism $\phi:e_iA \ra \omega \cM$. Now $g$ annihilates $\phi(e_i)$ so $\phi$ factors through a non-zero $e_iB \ra \omega \cM$. The corresponding morphism of sheaves $\mathcal{B}_i \ra \cM$ is injective since $\mathcal{B}_i$ is simple by Proposition~\ref{prop.Bisimple}. This proves part (1) and part (2) follows.

We now prove part (3). Let $\cM$ be a simple $g$-torsion sheaf. Now $\cM g$ is a proper submodule so must be 0 by simplicity. Part (1) gives a non-zero embedding of $\mathcal{B}_0$ into $\cM$ so we are done.
\end{proof}

We can now prove the main results of this subsection.
\begin{cor} \label{cor.uniserial}
The category ${\sf T}$ is uniserial.
\end{cor}

\begin{proof}
By \cite[8.3]{gabriel}, it suffices to show that each object of ${\sf T}$ has finite length, ${\sf T}$ has a unique simple object $\mathcal{S}$, and $\operatorname{Ext}^{1}(\mathcal{S},\mathcal{S})$ is one-dimensional as a left and right $\operatorname{End}(\mathcal{S})$-module.

We know from Proposition~\ref{prop.Biuniquesimple}(3) that $\mathcal{S} = \mathcal{B}_0$ is the unique simple in ${\sf T}$. From Serre duality (Proposition~\ref{prop.SerreAmodg}) and Proposition~\ref{prop.Biuniquesimple}(2), we have an isomorphism
$$\Ext^1(\mathcal{B}_0,\mathcal{B}_0) \cong \,\Hom(\mathcal{B}_0,\mathcal{B}_0)^*$$
of $(\End\,\mathcal{B}_0)-K$-bimodules. The calculation of $\End\,\mathcal{B}_0$ in Corollary~\ref{cor.endBi} then shows that $\Ext^1(\mathcal{B}_0,\mathcal{B}_0)$ must be one-dimensional as a left and right $\End\,\mathcal{B}_0$-module.

Now ${\sf T}$ is a finite length category since the length of any filtration of $\cM \in {\sf T}$ is bounded by the constant $f_{\omega \cM}$.
\end{proof}

\subsection{Classification of torsion sheaves}

From the point of view of non-commutative algebraic geometry, we can view our non-commutative projective line $\mathbb{P}^{nc}(V)$ as having a commutative subscheme consisting of a point with defining equation $g=0$ and an affine complement $\operatorname{Spec} A[g^{-1}]_{00}$. This decomposition allows us  to split the category of torsion sheaves mirroring Ringel's classification or regular $\Lambda$-modules in the case where $V$ is non-algebraic but non-simple. In this subsection we establish this splitting and identify the components.

If $M$ is $g$-torsion-free then we will say that $\pi M$ is {\em $g$-torsion-free}. We let ${\sf F}$ be the full subcategory of ${\sf Coh}\, \mathbb{P}^{nc}(V)$ consisting of torsion sheaves which are $g$-torsion free.

Let $M$ be a noetherian $A$-module and $M'$ be its $g$-torsion submodule. Then we have an exact sequence
$$ 0 \ra M' \ra M \ra M/M' \ra 0 $$
where $M/M'$ is $g$-torsion-free. This gives an analogous exact sequence of sheaves
\begin{equation}  \label{eqn.torsionseq}
0 \ra \cM' \ra \cM \ra \cM / \cM' \ra 0
\end{equation}
where $\cM  = \pi M, \cM' = \pi M'$.

\begin{thm}  \label{thm.splitcategory}
The category of torsion coherent sheaves on $\mathbb{P}^{nc}(V)$ is equivalent to the product category ${\sf T \times F}$.
\end{thm}
\begin{proof}
We first show that the exact sequence(\ref{eqn.torsionseq}) splits. This amounts to showing that $\Ext^1(\mathcal{N}, \cM') = 0$ for any $g$-torsion coherent sheaf $\cM'$ and $g$-torsion free coherent sheaf $\mathcal{N}$. From Corollary~\ref{cor.uniserial} we know that $\cM'$ is an iterated extension of $\mathcal{B}_0 = \mathcal{A}_0/\mathcal{A}_0g$ with itself. Hence it suffices to show that $\Ext^1(\mathcal{N}, \mathcal{B}_0) = 0$. Now $\Hom({\sf T, F}) = 0$ so Proposition \ref{prop.SerreAmodg} and Proposition \ref{prop.Biuniquesimple}(2) give
$$\Ext^1(\mathcal{N}, \mathcal{B}_0) = \Hom(\mathcal{B}_0, \mathcal{N})^* = 0.$$
It follows that the sequence (\ref{eqn.torsionseq}) splits and every torsion coherent sheaf is a direct sum of a $g$-torsion sheaf with a $g$-torsion free sheaf.

It remains only to show that $\Hom({\sf F, T}) = 0$, or equivalently, that $\Hom({\sf F, \mathcal{B}_0}) = 0$. Let $\mathcal{N}$ be a $g$-torsion free torsion coherent sheaf as before. Then $N = \omega \mathcal{N}$ is also $g$-torsion free so the multiplication by $g$ map  $N^{\phi^{-1}}(-\delta) \ra N$ is injective. It is also surjective since the Hilbert function $f_N$ is constant by Corollary~\ref{cor.hilbconstant} and we see $N/Ng = 0$. Now any homomorphism $N \ra \omega \mathcal{B}_0$ factors through $N/Ng$ so we are done.
\end{proof}
In the course of the proof above, we established the following fact.

\begin{lemma}  \label{lemma.omegagtorfree}
Let $\mathcal{N}$ be a $g$-torsion free torsion coherent sheaf on $\mathbb{P}^{nc}(V)$ and $N = \omega \mathcal{N}$. Then multiplication by $g$ on $N$ is a vector space isomorphism.
\end{lemma}

We finish this subsection by identifying ${\sf F}$ with a module category. To this end, consider the $\mathbb{Z}$-indexed algebra $A[g^{-1}]$. Let ${\sf Gr} A[g^{-1}]$ be the category of graded $A[g^{-1}]$-modules and ${\sf gr} A[g^{-1}]$ be its full subcategory of noetherian objects. Now $A[g^{-1}]$ is strongly graded, so we immediately obtain the following.

\begin{lemma}  \label{lemma.Aginvstrong}
There is a category equivalence ${\sf mod}\, A[g^{-1}]_{00} \cong {\sf gr}\, A[g^{-1}]$.
\end{lemma}

We need one more lemma.

\begin{lemma} \label{lemma.Aginfinite}
Let $0 \neq  x \in A_{i,i+1}$ which is linearly independent from $g_i$ if $\delta = 1$. Then $M:=e_i A/xA$ is $g$-torsion free. In particular, $e_iA[g^{-1}]$ has infinite length as an $A[g^{-1}]$-module.
\end{lemma}
\begin{proof}
Recall that $A$ is a domain so the Hilbert function $f_{M} = 1$ in degrees $\geq i$.  In addition, by Lemma~\ref{lemma.pd2}, $\underline{\Hom}_A(\bar{A}, M) = 0$ since $\operatorname{pd}\, M = 1$. Thus, $\tau M=0$.  Now by choice of $x$, we know that $g_i \notin xA$ so $Mg$ must also be a module whose Hilbert function is 1 in degrees $\geq i+\delta$. This proves that $M$ is indeed $g$-torsion free.   It follows that $xA[g^{-1}]$ is a submodule of $e_iA[g^{-1}]$ of colength 1. We may of course, continue finding submodules of colength 1 in $xA[g^{-1}]$ and so forth to show that $e_iA[g^{-1}]$ has infinite length.
\end{proof}

Let ${\sf fl}\, A[g^{-1}]_{00}$ denote the category of finite length objects in ${\sf mod}\, A[g^{-1}]_{00}$.

\begin{proposition} \label{prop.Fasmodulecat}
The category equivalence of Lemma~\ref{lemma.Aginvstrong} induces an equivalence of categories ${\sf fl}\, A[g^{-1}]_{00} \cong {\sf F}$.
\end{proposition}
\begin{proof}
We first identify the finite length subcategory of ${\sf Proj}\, A[g^{-1}]$. Note that there is a natural localization functor ${\sf Proj}\, A \ra {\sf Proj}\, A[g^{-1}] = {\sf Gr}\, A[g^{-1}]$. Now if $\cM$ is a torsion $g$-torsion free coherent sheaf on $\mathbb{P}^{nc}(V)$, then its image in  ${\sf Proj}\, A[g^{-1}]$ also has finite length. Conversely, given a finite length $A[g^{-1}]$-module $\tilde{M}$, we can find a noetherian $g$-torsion free $A$-module $M$ such that $\tilde{M} = M[g^{-1}]$. Using Theorem~\ref{thm.Grothendieck} and Proposition~\ref{prop.torsionsplitsoff}, we may assume that $M = M_f \oplus M_t$ where $M$ is graded free and $M_t$ is torsion. If $M_f\neq 0$ then its image in ${\sf Proj}\, A[g^{-1}]$ has infinite length by Lemma~\ref{lemma.Aginfinite}. Hence the finite length objects of  ${\sf gr}\, A[g^{-1}]$ are precisely those of the form $M[g^{-1}]$ where $M$ is a torsion $g$-torsion free noetherian $A$-module.

Let $\cM, \mathcal{N}$ be a $g$-torsion free torsion coherent sheaves on $\mathbb{P}^{nc}(V)$. We wish to show that the functor $\cM \mapsto \omega \cM [g^{-1}]$ is an equivalence of ${\sf F}$ with the finite length subcategory of ${\sf gr}\, A[g^{-1}]$. To this end, let $M = \omega \mathcal{M}, N = \omega\mathcal{N}$. Lemma~\ref{lemma.omegagtorfree} shows that $M = M[g^{-1}], N= N[g^{-1}]$. Then
\begin{eqnarray*}
\Hom_{\mathbb{P}^{nc}(V)}(\cM,\mathcal{N}) & = & \Hom_A(M,N) \\
& = & \Hom_A(M[g^{-1}],N[g^{-1}]) \\
& = & \Hom_{A[g^{-1}]}(M[g^{-1}],N[g^{-1}]).
\end{eqnarray*}
Combining with Lemma~\ref{lemma.Aginvstrong} yields the proposition.
\end{proof}

We next show that the non-commutative ``affine co-ordinate ring'' $A[g^{-1}]_{00}$ is well understood when the bimodule $V$ is non-simple. We use the following notation: given a ring $R$ and automorphism $\alpha$, we let $R[z;\alpha]$ denote the skew polynomial ring with defining relation $zr = \alpha(r) z$, whilst if $\partial$ is a derivation of $R$, we let $R[z;\partial]$ be the skew polynomial ring with defining relation $zr = rz + \partial{r}$.

\begin{proposition}  \label{prop.oreext}
Suppose $V$ is non-simple so we are either in the split case (1a) or non-split case (1b) of Lemma~\ref{lemma.twosidedclass}. In the notation of that lemma, either
\begin{enumerate}
\item $V$ is split, in which case $\sigma\tau^{-1}$ is an automorphism of $K_0$ and $A[g^{-1}]_{00} \cong K_0[z; \sigma\tau^{-1}]$, or
 \item $V$ is non-split, in which case $\delta\sigma^{-1}$ is a derivation of $K_0$ and $A[g^{-1}]_{00} \cong K_0[z; \delta\sigma^{-1}]$.
\end{enumerate}
\end{proposition}
\begin{proof}
In both cases, one verifies directly that $A[g^{-1}]_{00}$ is generated as a ring by $K_0$ and $z = xy^{-1}$ in the notation of Section~\ref{section.easytwist}. The defining relations for the skew polynomial ring are easily checked for $A[g^{-1}]_{00}$. For example, if $V$ is non-split  and $r \in K_0$ then
$$ xy^{-1} r = x \sigma^{-1}(r) y^{-1} = (rx + \delta \sigma^{-1}(r)y)y^{-1} = rxy^{-1}  + \delta \sigma^{-1}(r).$$
It remains only to show that the $z^i$ are left linearly independent over $K_0$. The easiest way to see this is to consider the injective right multiplication by $(g_0g_1\ldots g_{n-1})^{-1}$ map $\rho: A_{0n} \hookrightarrow A[g^{-1}]_{00}$. The result follows since the left dimension $\dim_{K_0} A_{0n} = n+1$ whilst $\operatorname{im} \rho$ is spanned over $K_0$ by $1,z, \ldots, z^n$.
\end{proof}

\subsection{Classification of regular modules} \label{subsec.regular}
In this subsection, we use the classification results for torsion sheaves to give an explicit description of regular $\Lambda$-modules. Putting together Corollary~\ref{cor.regularistorsion}, Theorem~\ref{thm.splitcategory} and Proposition~\ref{prop.Fasmodulecat} we immediately obtain the following description of the category of regular $\Lambda$-modules.

\begin{thm}  \label{thm.regularmod}
The category of regular $\Lambda$-modules is equivalent to the product category ${\sf fl} A[g^{-1}]_{00} \times {\sf T}$ where ${\sf T}$ is uniserial.
\end{thm}

We now use the derived equivalence $\Phi$ to give a description of these regular modules. We start with the ones corresponding to ${\sf T}$, the $g$-torsion sheaves.

Let $\mathcal{B}_0$ be the unique simple $g$-torsion sheaf. It's quite easy to determine $\Phi(\mathcal{B}_0)$. If $V$ is simple so is given by a field $F$ in the notation of Lemma~\ref{lemma.twosidedclass}(2), then $\Phi(\mathcal{B}_0)$ is the $\Lambda$-module
$$ m: F \otimes_{K_1} F \ra F $$
where $m$ is field multiplication. If $V$ is non-simple, then $g_0 \in V$ generates a proper sub-bimodule $Kg_0$ and the quotient $\bar{V} := V/ Kg_0$ is an invertible $K-K$-bimodule. In this language, $\Phi(\mathcal{B}_0)$ is the natural map
$$ K \otimes_K V \ra K \otimes_K \bar{V} \cong K .$$
Consider now the indecomposable $n$-fold extension $\mathcal{B}_0(n)$ of $\mathcal{B}_0$ with itself. The following proposition allows us to ``read off'' $\Phi(\mathcal{B}_0(n))$ from the non-commutative symmetric algebra.
\begin{proposition}  \label{prop.phiT}
Let $g^n = g_{1-n\delta}g_{1-(n-1)\delta}\ldots g_{1-\delta}$. Then
$\Phi(\mathcal{B}_0(n))$ is given by the composite map
$$ A_{1-n\delta,0} \otimes_K V \xrightarrow{m} A_{1-n\delta,1} \xrightarrow{q} \frac{A_{1-n\delta,1}}{g^nA_{11}}$$
where $m$ is multiplication in $A$ and $q$ is the canonical quotient map.
\end{proposition}
\begin{proof}
Consider the exact sequence
$$ 0 \ra e_{1}A \xrightarrow{g^n} e_{1-n\delta}A \ra e_{1-n\delta}A/g^n A \ra 0 .$$
Now $\mathcal{B}_0(n) = \pi(e_{1-n\delta}A/g^n A)$. Hence $\Phi(\mathcal{B}_0(n)$ is given by
$$ \omega\pi(e_{1-n\delta}A/g^n A)_0 \otimes_K V \ra \omega\pi(e_{1-n\delta}A/g^n A)_1 .$$
Also,  $\operatorname{pd} e_{1-n\delta}A/g^n A = 1$ so $e_{1-n\delta}A/g^n A$ embeds in $\omega(e_{1-n\delta}A/g^n A)$ by Lemma~\ref{lemma.pd2}. Furthermore, the Hilbert function of $e_{1-n\delta}A/g^n A$ is the constant $n\delta$ in non-negative degrees so in fact $(e_{1-n\delta}A/g^n A)_i = \omega\pi(e_{1-n\delta}A/g^n A)_i$ for $i \geq 0$.
\end{proof}

We turn now to studying the regular modules corresponding to $g$-torsion free sheaves. To simplify notation, we let $A_g := A[g^{-1}]_{00}$. Note that the multiplication map $A_g \otimes_K V \ra A[g^{-1}]_{01}$ is surjective, so we will usually abbreviate $A[g^{-1}]_{01}= A_g V$. Let $\Psi: {\sf fl} A_g \xrightarrow{\cong} {\sf F}$ be the category equivalence of Proposition~\ref{prop.Fasmodulecat}. The proof of this proposition immediately gives
\begin{proposition}  \label{prop.phipsi}
The composite functor $\Phi\Psi$ sends a finite length $A_g$-module $M$ to the $\Lambda$-module
$$ m: M \otimes_K V \ra M \otimes_{A_g} A_g V $$
where $m$ is induced by inclusion $V \hookrightarrow A_g V$.
\end{proposition}
To make the $\Lambda$-module in this proposition explicit, we need only compute the left $A_g$-module $A_gV$. We give a minimal presentation of this module which, by right exactness of $\otimes$ will allow us to compute explicitly the term $M \otimes_{A_g} A_g V $ above. Given a two-sided basis $x,y$ for $V$ we let $^*x,^*y$ be a dual basis in $^*V$ and $^{**}x,^{**}y$ be a dual basis of that in $^{**}V$.

\begin{proposition}  \label{prop.AgV}
If $V$ is non-simple, then $A_g V \cong A_g$ and in fact is freely generated by $y$. If $V$ is simple, then there is a presentation
$$A_g^{\oplus 2}
\xrightarrow{\begin{pmatrix} g_{-2}^{-1} \,^{**}x \, ^*x& g_{-2}^{-1} \,^{**}x \, ^*y \\ g_{-2}^{-1} \,^{**}y \, ^*x & g_{-2}^{-1} \,^{**}y \, ^*y \end{pmatrix}}
A_g^{\oplus 2}
\xrightarrow{\begin{pmatrix} x \\ y \end{pmatrix}} A_gV \ra 0.$$
\end{proposition}
\begin{proof}
Let $x,y$ be a two-sided basis for $V$. We know $x,y$ then also generates the $A_g$-module $A_gV$. It remains then only to compute module relations. Any such relation can be written in the form
$$ g^{-n}a x + g^{-n} b y = 0 $$
for some $a,b \in A_{-n\delta ,0}$ and $g^{-n} = g_{-\delta}^{-1}g_{-2\delta}^{-1} \ldots g_{-n\delta}^{-1}$. This gives the relation $ax+by=0$ in $A$. The Euler exact sequence (Lemma~\ref{lemma.eulerses}(2)) shows that $(a \ b) = (c\,^*x \ c\, ^*y)$ for some $c \in A_{-n\delta,-1}$. If $\delta = 1$, then the relations are generated by $(a \ b) = (\,^*x \ \, ^*y)$. In this case, we may furthermore assume that $g_{-1} = \,^*x$ so the relation reduces to $x = - (g^{-1}  \,^*y) y$ which completes the cases when $V$ is non-simple. When $V$ is simple, then $\delta=2$ so $c$ has to have odd degree and the relations are given by letting $c$ vary over a basis of $A_{-2,-1}$. This completes the case where $\delta = 2$.
\end{proof}
Putting together these two propositions gives a classification of regular modules corresponding to $g$-torsion free sheaves which generalizes exactly the construction Ringel gives in the $V$ non-simple case in \cite[Theorem~7.4]{species}.


\begin{thebibliography}{11}

\bibitem{asz} K. Ajitabh, S. P. Smith, and J. J. Zhang. Auslander-Gorenstein rings, {\it Comm. Algebra} {\bf 26} (1998), 2159-2180.


\bibitem{atv1} M. Artin, J. Tate, and M. Van den Bergh, Some algebras associated to automorphisms of elliptic curves, in {\it The Grothendieck Festschrift, Vol. 1}, 33-85, Birkhäuser, Boston, 2007.

\bibitem{az} M. Artin and J. Zhang, Noncommutative projective schemes, {\it Adv. Math.} {\bf 109} (1994), 228-287.

\bibitem{chan} D. Chan and A. Nyman, Non-commutative Mori contractions and $\mathbb{P}^{1}$-bundles, {\it Adv. Math.} {\bf 245} (2013), 327-381.


\bibitem{dlabringel}  V. Dlab and C. M. Ringel, Indecomposable representations of graphs and algebras, {\it Mem. Amer. Math. Soc.} {\bf 6} (1976). no. 173.

\bibitem{dlabringel2} V. Dlab and C. M. Ringel, The preprojective algebra of a modulated graph, {\it Representation theory, II}, pp. 216-231, Lecture Notes in Math., 832, {it Springer, Berlin-New York}, 1980.

\bibitem{gab} P. Gabriel, Des cat\'{e}gories ab\'{e}liennes, {\it Bull. Soc. Math. France} {\bf 90} (1962), 323-448.

\bibitem{gabriel} P. Gabriel, Indecomposable representations II,  {\it Symposia Mathematica, Vol. XI}, pp. 81-104. {\it Academic Press, London}, 1973.

\bibitem{hart} J. Hart and A. Nyman, Duals of simple two-sided vector spaces, {\it Comm. Algebra} {\bf 40} (2012), 2405-2419.

\bibitem{harts} R. Hartshorne, Algebraic geometry.  Graduate Texts in Mathematics, No. 52. {\it Springer-Verlag, New York-Heidelberg}, 1977.

\bibitem{keller} B. Keller, Derived categories and tilting.  {\it Handbook of tilting theory,} 49-104, London Math. Soc. Lecture Note Ser., 332 {\it Cambridge Univ. Press, Cambridge, } 2007.


\bibitem{levasseur} T. Levasseur,  Some properties of non-commutative regular graded rings, {\it Glasgow Math. J.} {\bf 34} (1992), 277-300.

\bibitem{mcr} J. C. McConnell and J. C. Robson, Noncommutative noetherian rings.  With the cooperation of L. W. Small. Revised edition.  Graduate Studies in Mathematics, 30. {\it American Mathematical Society, Providence, RI}, 2001.


\bibitem{tsen} A. Nyman, Noncommutative Tsen's theorem in dimension one, {\it J. Algebra} {\bf 434} (2015), 90-114.

\bibitem{duality} A. Nyman, Serre duality for non-commutative $\mathbb{P}^{1}$-bundles, {\it Trans. Amer. Math. Soc.} {\bf 357} (2005), 1349-1416.

\bibitem{finiteness} A. Nyman, Serre finiteness and Serre vanishing for non-commutative $\mathbb{P}^{1}$-bundles, {\it J. Algebra} {\bf 278} (2004), 32-42.

\bibitem{nyman} A. Nyman, The geometry of arithmetic noncommutative projective lines, {\it J. Algebra} {\bf 414} (2014), 190-240.

\bibitem{witt} A. Nyman, Witt�s theorem for noncommutative conics, {\it Appl. Categ. Structures}, to appear.


\bibitem{P}  D. Patrick, Noncommutative symmetric algebras of
two-sided vector spaces, {\it J. Algebra} {\bf 233} (2000), 16-36.


\bibitem{presotto} D. Presotto and Louis de Thanhoffer de V\"{o}lcsey, Homological properties of a certain noncommutative Del Pezzo surface, arXiv:1503.03992, 2015.

\bibitem{crem} D. Presotto and M. Van den Bergh, Noncommutative versions of some classical birational transformations, arXiv:1410.5207v2, 2014.

\bibitem{species} C. Ringel, Representations of $K$-species and bimodules, {\it J. Algebra} {\bf 41} (1976), 269-302.

S. P. Smith, Integral non-commutative spaces, {\it J. Algebra} {\bf 246} (2001), 793-810.

\bibitem{smithnotes}
S. P. Smith, Non-commutative Algebraic Geometry, {\it unpub. notes} (2000).

\bibitem{vandenbergh}
M. Van den Bergh, Non-commutative $\mathbb{P}^{1}$-bundles over commutative schemes, {\it Trans. Amer. Math. Soc.} {\bf 364} (2012), 6279-6313.

\bibitem{quadrics}
M. Van den Bergh, Noncommutative quadrics, {\it Int. Math. Res. Not.} (2011), no. 17, 3983-4026.

\bibitem{weibel} C. Weibel, An introduction to homological algebra. Cambridge Studies in Advanced Mathematics, {\it Cambridge University Press, Cambridge}, 1994.

\bibitem{yz} A. Yekutieli and J. Zhang. Rings with Auslander dualizing complexes. {\it J. Algebra} {\bf 213 } (1999), 1-51.

\end{thebibliography}
\end{document}